\documentclass[10pt,a4paper,twoside]{amsart}
\usepackage{hyperref}
\usepackage[english]{babel}
\usepackage{enumerate,colonequals,expdlist}
\usepackage{amssymb,amsmath,amsfonts,amsthm,mathtools}
\usepackage{color,graphicx}
\usepackage{tikz,float}
\usetikzlibrary{patterns}
\renewcommand{\epsilon}{\varepsilon}

\newcommand{\RR}{\mathbb{R}}
\newcommand{\CC}{\mathbb{C}}
\newcommand{\NN}{\mathbb{N}}
\newcommand{\ZZ}{\mathbb{Z}}

\newcommand{\TT}{\mathbb{T}}

\DeclareMathOperator{\supp}{supp}

\newcommand{\Ran}{\operatorname{Ran}}

\newcommand{\sfuc}{\mathrm{uc}}
\newcommand{\abs}[1]{\lvert{#1}\rvert}    
\newcommand{\normsymb}{\|}
\newcommand{\norm}[2]{\normsymb{#1}\normsymb_{#2}}  

\newcommand{\h}{\widehat} 
\usepackage{scalerel,stackengine}
\stackMath
\newcommand\reallywidehat[1]{%
\savestack{\tmpbox}{\stretchto{%
\scaleto{%
\scalerel*[\widthof{\ensuremath{#1}}]{\kern-.6pt\bigwedge\kern-.6pt}%
{\rule[-\textheight/2]{1ex}{\textheight}}
}{\textheight}%
}{0.5ex}}%
\stackon[1pt]{#1}{\tmpbox}%
}
\DeclareMathOperator{\dd}{\, d}

\newtheorem{theorem}{Theorem}[section]
\newtheorem{thm}[theorem]{Theorem}
\newtheorem{lemma}[theorem]{Lemma}

\newtheorem*{assumption*}{Assumption}
\newtheorem{proposition}[theorem]{Proposition}

\newtheorem*{corollary*}{Corollary}
\newtheorem{question}[theorem]{Question}

\theoremstyle{definition}
\newtheorem{definition}[theorem]{Definition}
\theoremstyle{remark}
\newtheorem{remark}[theorem]{Remark}
\newtheorem{example}[theorem]{Example}

\begin{document}
\title[Unique continuation \& Logvinenko-Sereda Theorems on the torus]{Scale-free unique continuation estimates and Logvinenko-Sereda Theorems on the torus}
\author[M.~Egidi]{Michela Egidi}
\address[M.~E.]{Ruhr-Universi\"at Bochum, Germany}
\author[I.~Veseli\'c]{Ivan Veseli\'c}
\address[I.~V.]{TU Dortmund, Germany}

\thanks{
{\today, \jobname.tex}. \\
See also Annales Henri Poincar\'e, https://dx.doi.org/10.1007/s00023-020-00957-7} 

\thanks{}


\begin{abstract}
We study uncertainty principles for function classes on the torus.
The classes are defined in terms of spectral subspaces of the energy or the momentum, respectively.
In our main theorems, the support of the Fourier transform of the considered functions is allowed to be contained
in (a finite number of) $d$-dimensional cubes.
The estimates we obtain do not depend on the size of the torus and the position of the $d$-dimensional cubes,
but only on their size and number, and the
density and scale of the observability set.
Our results are on the one hand closely related to unique continuation for linear combinations of eigenfunctions (aka spectral inequalities)
which can be obtained by Carleman estimates, on the other hand to observability estimates for the time-dependent Schr\"odinger and for the heat equation,
and finally to the Logvinenko \& Sereda theorem.
In fact, they are based on the  methods developed by Kovrijkine to refine and generalize the results of Logvinenko \& Sereda and Kacnel'son.
Furthermore, relying on completely different techniques associated with the time-dependent Schr\"odinger equation, we prove a companion theorem where the energy of the considered functions is allowed to be in a spectral subspace of a Schr\"odinger operator.
\end{abstract}

\maketitle

\section{Motivation and history}\label{sec:intro}

We study $L^2$-equidistribution properties of functions on a torus which are uniform over a (properly chosen) linear subspace.
Since we are in particular interested in very large tori, we include in our discussion also the case when the domain of the functions is the whole of $\RR^d$.
Here by  $L^2(\Omega )$-equidistribution, for some measurable $\Omega  \subset \RR^d$, we mean that the squared total norm $\int_\Omega  |f|^2$
is controlled by the squared $L^2$-norm $\int_S |f|^2$  over a subset $S \subset \Omega $
provided that $S$ is in a sense (which will be made precise below) evenly distributed within $\Omega $.

The main body of this paper is concerned with the case that the linear subspace consists of functions given by Fourier series
with Fourier coefficients vanishing outside an explicitly prescribed bounded region in momentum space.
The physical interpretation is that the function
describes a quantum state where the momentum vector is localized in the specified set.
A companion result studies the case that the linear subspace
is the range of a spectral projector of a Schr\"odinger operator corresponding to a compact energy interval.
Since a Schr\"odinger operator is the observable operator corresponding to the energy,
this means that in this case the range of values of the energy is localized, in contrast to the specification of the momentum in the first case.

In this section we present several results in the literature that led to questions which motivated this paper.

Let us first consider the case of functions whose Fourier transform is  supported in a bounded region.
For $\Omega =\RR^d$, this question goes back at least to the work of Panejah \cite{Panejah-61, Panejah-62},
which was then independently generalized by Logvinenko \& Sereda \cite{LogvinenkoS-74} and  Kacnel'son \cite{Kacnelson-73} .
In order to spell out their result we need a geometric definition.
\begin{definition}\label{defin:thickness}
Let $d\geq 1$ and let $S$ be a measurable subset of $\RR^d$. We say that $S$ is a \emph{thick set} if there exist
$\gamma\in(0,1]$ and $a=(a_1,\ldots,a_d)\in \RR^d_+$ such that
\begin{equation}\label{def:thick-set}
\forall \ x \in \RR^d: \quad \abs{(S +x)\cap ([0,a_1]\times \ldots \times [0,a_d])}\geq\gamma \prod_{j=1}^d a_j
\end{equation}
Here $\abs{\cdot}$ denotes the Lebesgue measure.
We will refer to $S$ as $(\gamma,a)$-\emph{thick} to emphasise the parameters, and to $(a_1,\ldots,a_d)$
as the \emph{size} of the window  $[0,a_1]\times \ldots \times [0,a_d]$.
\end{definition}
One sees that in particular the complement of a thick set cannot contain arbitrary large balls.

In the following we will adopt the \emph{following convention} to make the dependency of constants explicit:
We call $C$ a numerical constant
if it is a real number independent of any parameter appearing in this paper.
We write $C=C(\phi, \zeta)$ if $C$ is a real constant depending on parameters
$\phi$ and $\zeta$ but independent of any other parameter (appearing in this paper).

The results of \cite{LogvinenkoS-74}, \cite{Kacnelson-73} imply in particular the following
\begin{theorem}[Logvinenko \& Sereda, Kacnel'son ]\label{thm:LS-Kat}
Let $d\in\NN$, $b>0$, $p\in [1,\infty)$, and $S\subset\RR^d$ be a measurable set.
Then the following statements are equivalent:
\begin{itemize}
 \item
 $S$ is a thick set;
 \item
 for all $f\in L^p(\RR^d)$ such that $\supp\widehat{f}\subset B(0,b)$ there exists a constant $C=C(S,b)>0$ such that
\begin{equation}\label{eq:LS-Kat}
 \norm{f}{L^p(\RR^d)}\leq C\norm{f}{L^p(S)}.
\end{equation}
\end{itemize}
\end{theorem}

This result was then sharply quantified by Kovrijkine, who also extended it to functions
with Fourier transform supported in a finite union of $d$-dimensional rectangles.

\begin{definition}\label{def:d-rectangle}
A subset $J \subset \RR^d$ is called \emph{$d$-dimensional rectangle}, if there are
$\lambda_{1}\ldots \lambda_{d}\in \RR$ and $b_1\ldots b_d\in (0,\infty)$ such that
\begin{equation}\label{eq:d-rectangle}
J = [\lambda_{1}- \frac{b_1}{2},\lambda_{1}+ \frac{b_1}{2}] \times\ldots\times
[\lambda_{d}- \frac{b_d}{2},\lambda_{d}+ \frac{b_d}{2}] \quad
\end{equation}
The $b_1,\ldots, b_d$ are called \emph{sidelengths} of $J$.
\end{definition}

\begin{theorem}[Kovrijkine \cite{Kovrijkine-thesis,Kovrijkine-01}]\label{thm:kovrijkine}
 Let $d,n\in\NN$, $p\in[1,\infty]$, and let $S$ be a $(\gamma, a)$-thick set in $\RR^d$. Let also $J$, $J_1,\ldots, J_n$ be
 $d$-dimensional rectangles, each with sidelengths $b_1,\ldots, b_d$. Set $b=(b_1,\ldots,b_d)$.
\begin{itemize}
 \item [(I)] If $f\in L^p(\RR^d)$ such that $\supp\widehat{f}\subset J$, then
\begin{equation} \label{eq:kovrijkine-1}
 \norm{f}{L^p(\RR^d)}\leq \left(\frac{C_1^d}{\gamma}\right)^{C_1(a\cdot b+d)} \norm{f}{L^p(S)},
\end{equation}
where $C_1>0$ is a numerical constant.
\item [(II)] If $f\in L^p(\RR^d)$ such that $\supp\widehat{f}\subset J_1\cup\ldots\cup J_n$, then
\begin{equation}\label{eq:kovrijkine-2}
\norm{f}{L^p(\RR^d)}\leq\left(\frac{C_2^d}{\gamma}\right)^{\left(\frac{C_2^d}{\gamma}\right)^n a\cdot b +n -\frac{p-1}{p}}\norm{f}{L^p(S)},
\end{equation}
where $C_2>0$ a numerical constant.
\end{itemize}
Here $a\cdot b$ denotes the Euclidean inner product in $\RR^d$.
\end{theorem}

Theorem \ref{thm:kovrijkine} lends a precise mathematical formulation to the physical \emph{uncertainty principle} that a quantum state cannot be simultaneously localized in position variable and momentum space.
Note that there are various other quantitative mathematical implementations of this principle, for instance Heisenberg's uncertainty relation,
Hardy's uncertainty principle, the Paley--Wiener theorem, or the framework of annihilating pairs, which we formulate next, restricting ourselves to $p=2$.
Let $\mathcal{E},\mathcal{B}\subset\RR^d$. We say that $(\mathcal{E}, \mathcal{B})$
is a \emph{strong annihilating pair} (see, e.g. \cite[Chapter 3]{HavinJ-94}) if there exists a constant
$C=C(\mathcal{E},\mathcal{B})$ such that
\begin{equation}\label{eq:annihilating-pair}
 \forall\, f\in L^2(\RR^d)  \quad \norm{f}{L^2(\RR^d)}^2\leq C \left( \norm{f}{L^2(\mathcal{E}^c)}^2 +\norm{\widehat{f}}{L^2(\mathcal{B}^c)}^2\right)   .
\end{equation}
Then, Theorem \ref{thm:kovrijkine} above and \cite[p. 88, Subsection 1 A)]{HavinJ-94}
imply that the pair $(S^c, J_1\cup\ldots\cup J_n)$ is strong annihilating, cf.~also the argument in Remark \ref{rmk:annihilating-pair}.
For results about strong annihilating pairs we refer the reader to \cite{GhobberJ-13, HavinJ-94, Jaming-07, Nazarov-94}
and the references therein.

One aim of this paper is to derive, using techniques developed in \cite{Kovrijkine-thesis},  estimates corresponding to \eqref{eq:kovrijkine-1} and \eqref{eq:kovrijkine-2} in the case where $\Omega=\RR^d$ is replaced by
the $d$-dimensional torus $\Omega=\TT_L^d$  with sides of length $2\pi L_1, \ldots, 2\pi L_d$.
The motivation for this goal stems
originally from the objective to establish precise uncertainty relations for spectral projections of Schr\"odinger operators
on  $d$-dimensional cubes (see the subsequent theorem and discussion) as they are studied in particular in the theory of random Schr\"odinger operators, see~e.g.~\cite{CombesHK-03,GerminetK-13,RojasMolinaV-13,NakicTTV-18}
and the references therein.
From this context stems also the interest in deriving estimates which are valid for all (sufficiently large) $L\in \RR_+^d$ and are uniform in $L$.
For this reason, we consider `observability sets' $S$  which are subsets of $\RR^d$, rather than subsets of one fixed  $\TT_L^d$,
and are  $(\gamma, a)$-thick. In order to `see' some part of $S$ in the torus $\TT_L^d$, it is necessary to assume that the size of the torus is larger than the size of the window,
i.e.~$a_1\leq 2\pi L_1, \ldots, a_d\leq 2\pi L_d$.

Now we turn to the second topic of the paper, concerning equidistribution properties of functions belonging to a subspace associated with a compact energy interval.
Specifically, we recall a result proven in \cite{NakicTTV-18,NakicTTV-18-JST}, that applies to the Schr\"odinger operator $ H_\Omega = -\Delta_\Omega + V_\Omega $ on $L^2 (\Omega)$.
Here we assume that $\Omega$ is a generalized ($d$-dimensional) rectangle, i.\,e.
\begin{equation}\label{eq:Omega}
\Omega = \bigtimes_{k =1}^d (\alpha_i , \beta_i) , \quad \alpha_i ,\beta_i \in \RR\cup \{\pm \infty\}, \alpha_i < \beta_i \quad \text{ for all } i=1,\ldots, d.
\end{equation}
Note that in contrast to a usual $d$-dimensional rectangle the sidelengths may be infinite.
We also assume that  $\Omega \supset \Lambda_G:=(-G/2 , G/2)^d$ where $G=\min\{\beta_1-\alpha_1,\ldots ,\beta_d-\alpha_d \}$.
(This can be done without loss of generality, since there always exists \emph{some} $x_0\in \RR^d$ such that $\Omega \supset \Lambda_G+x_0$,
and we can perform a global shift of coordinates to achieve $x_0=0$.)
In the Schr\"odinger operator above  $\Delta_\Omega $ denotes the Laplacian with Dirichlet or  Neumann boundary conditions
and $V_\Omega \colon \Omega \to \RR$ a measurable and bounded potential.
In the case that $\Omega=\Lambda_L := (-L/2 , L/2)^d \subset \RR^d$ is a cube with sidelengths $L>0$, $H_\Omega$ may also be equipped with periodic boundary conditions. This relates to functions on a torus, as will be made precise below.
To formulate the main results of \cite{NakicTTV-18,NakicTTV-18-JST}, we need another geometric notion.
\begin{figure}[ht]\centering
\begin{tikzpicture}
\pgfmathsetseed{{\number\pdfrandomseed}}
\foreach \x in {0.5,1.5,...,4.5}{
\foreach \y in {0.5,1.5,...,4.5}{
\filldraw[fill=gray!70] (\x+rand*0.35,\y+rand*0.35) circle (0.15cm);
}
}
\foreach \y in {0,1,2,3,4,5}{
\draw (\y,0) --(\y,5);
\draw (0,\y) --(5,\y);
}

\begin{scope}[xshift=-6cm]
\foreach \x in {0.5,1.5,...,4.5}{
\foreach \y in {0.5,1.5,...,4.5}{
\filldraw[fill=gray!70] (\x,\y) circle (1.5mm);
}
}
\foreach \y in {0,1,2,3,4,5}{
\draw (\y,0) --(\y,5);
\draw (0,\y) --(5,\y);
}
\end{scope}
\end{tikzpicture}
\caption{Illustration of $S_{\delta,\Omega}$ within the region $\Omega = \Lambda_5 \subset \mathbb{R}^2$ for $G=1$, and periodically (left)
and non-periodically (right) arranged balls, respectively.\label{fig:equidistributed}}
\end{figure}

\begin{definition}\label{def:equidistributed}
Let $G > 0$ and $0 < \delta < G/2$.
We say that a sequence $Z=\{z_j\}_{j \in \ZZ^d}$, is \emph{$(G,\delta)$-equidistributed} if
\begin{equation}\label{eq:equidistributed}
\forall j \in \ZZ^d \colon \quad  B(z_j , \delta) \subset G(\Lambda_1 + j) .
\end{equation}
Given a $(G,\delta)$-equidistributed sequence $Z$ and a generalized rectangle $\Omega$, we define the set
\begin{equation}\label{eq:eq:28}
S_{\delta , \Omega} := \bigcup_{j \in \ZZ^d } B(z_j , \delta) \cap \Omega,
\end{equation}
see Fig.~\ref{fig:equidistributed} for an illustration.
We suppress the dependence of  the set $S_{\delta,\Omega}$ on $G$ and the choice of the $(G,\delta)$-equidistributed sequence $Z=\{z_j\}_{j}$ in the notation.
\end{definition}

If a sequence $Z$ is $(G,\delta)$-equidistributed,
the set $S_{\delta} = \bigcup_{j \in \ZZ^d } B(z_j , \delta) $ is $(\gamma, a)$-thick with
$\gamma= \omega_d\delta^d$ and $a=(2G, \ldots,2G)$, where $\omega_d$ is the volume of the unit ball in $d$ dimensions.
Thus the class of subsets $S \subset \RR^d$ considered in the following theorem is more restrictive than the one in Theorem \ref{thm:kovrijkine}.

\begin{theorem}[Naki\'c, T\"aufer, Tautenhahn, Veseli\'c \cite{NakicTTV-18,NakicTTV-18-JST}] \label{thm:NTTV}
There is $N>0$ depending only on $d$, such that
for all $G>0$, all generalized rectangles $\Omega \subset\RR^d$, with  $\Omega \supset\Lambda_G$, all $\delta \in (0,G/2)$, all $(G,\delta)$-equidistributed sequences $Z$,
all real-valued $V_{\Omega} \in L^\infty (\Omega)$, all $E \in \RR$, and all $f\in \Ran \chi_{(-\infty,E]} (H_\Omega)$ we have
\begin{equation} \label{eq:UCP}
\lVert f \rVert_{L^2 (S_{\delta , \Omega})}^2
\geq
C_\sfuc \lVert f \rVert_{L^2 (\Omega)}^2, \ \text{ where }\
 C_\sfuc = \sup_{\lambda \in \RR}
 \left(\frac{\delta}{G}\right)^{N \bigl(1 + G^{4/3}\lVert V_\Omega-\lambda \rVert_\infty^{2/3} + G\sqrt{(E-\lambda)_+} \bigr)}
\end{equation}
and $t_+:=\max\{0,t\} $ for $t \in \RR$.
\end{theorem}

The estimate is \emph{scale-free}, in the sense that we have the same constant for all generalized rectangles $\Omega \subset\RR^d$, with  $\Omega \supset\Lambda_G$, for $G>0$ fixed.
In particular, this means that the constant remains bounded if we choose the domain $\Omega$ to be any of the  $d$-dimensional cubes $\Lambda_{LG}, L\in\NN$.
Note that the Hamiltonian $H_\Omega$ in Theorem \ref{thm:NTTV} is lower bounded in the sense of quadratic forms by $-\|V_\Omega\|_\infty$, consequently
$\chi_{(-\infty,E]}(H_{\Omega})=\chi_{[-\|V_\Omega\|_\infty,E]}(H_{\Omega})$
corresponding to the compact energy interval $[-\|V_\Omega\|_\infty,E]$.

Now we are in the position to formulate a question which arises when one compares Theorems \ref{thm:kovrijkine} and \ref{thm:NTTV}
and tries to find a unified framework which would cover both of them.

First of all one sees that the two theorems have an overlap. Indeed if $H=-\Delta$ on $\Omega = \RR^d$
and $E>0$ then any  $f\in \Ran(\chi_{(-\infty,E]}(H))$ satisfies $\supp\widehat{f}\subset[-\sqrt{E}, \sqrt{E}]^d$
and \eqref{eq:kovrijkine-1}  in Theorem \ref{thm:kovrijkine} gives
\begin{equation}\label{eq:spectral-inequality}
 \norm{f}{L^2(\RR^d)}^2
 \leq \left(\frac{C_1^d}{\gamma}\right)^{C_1(2\sqrt{E}\abs{a}_1+d)}\norm{f}{L^2(S)}^2, \quad \text{ where } \abs{a}_1:=\sum_{j=1}^d a_j
\end{equation}
for any $(\gamma, a)$-thick set $S$.
In \eqref{eq:spectral-inequality} the prefactor is more precise and the allowed `observability set' $S$ is more general
than in Theorem \ref{thm:NTTV}, however no potential $V$ is allowed.

It is remarkable that the prefactor in \eqref{eq:kovrijkine-1} does not change if one translates the
$d$-dimensional rectangle $J$ in Fourier space.
Even more remarkably, one can translate the $d$-dimensional rectangles $J_1,\ldots, J_d$  in
\eqref{eq:kovrijkine-2} independently of each other, while still not changing the corresponding prefactor.
Motivated by this fact one could be optimistic and ask whether it is possible to prove a variant of Theorem \ref{thm:NTTV} where (for an appropriate width $w>0$) the prefactor $C_\sfuc$ can be chosen uniformly over spectral subspaces $\Ran (\chi_{[E-w, E]}(H))$, independently of the energy shift $E$.
If $V\equiv 0$ this subspace consists of functions with momentum contained in the shell $B(0,\sqrt{E})\setminus B(0,\sqrt{E-w})$.

The results we derive in the present paper concern the case that the configuration space
$\Omega$ equals a torus $\TT_L^d:=[0,2\pi L_1]\times\ldots\times[0,2\pi L_d]$,
hence we spell out a precise version of the question corresponding to this geometric situation.
Correspondingly, the Laplacian is equipped with periodic boundary conditions.
We abbreviate here $H_{\TT_L^d}$ by $H_L$ and $S_{\delta, \TT_L^d}$ by $S_{\delta, L}$.

\begin{question} \label{question}
Is it possible to find for all $d\in \NN$,
all $\delta>0$, 
all $(1,\delta)$-equidistributed sequences $Z$,
all measurable and bounded potentials $V\colon\RR^d\to \RR$,
all $w>0$
and
all $L\in \NN^d$,
a constant $C=C(d, \delta, V,w,L)$ such that
\begin{equation}\label{eq:question-E-independent-UCP}
\forall f\in \bigcup_{E\in \RR}\Ran (\chi_{[E-w, E]}(H)): \quad
\int_{S_{\delta , L}} |f|^2
\geq
C\int_{\TT_L^d}|f|^2 \quad ?
\end{equation}
\end{question}
We elaborate several facets of this question:
\begin{enumerate}[(a)]
  \item
  Note that the constant $C$ appearing in \eqref{eq:question-E-independent-UCP} is by definition independent of the energy $E$ appearing as the maximum of the interval $[E-w,E]$.
  \item
If there is a positive answer to this question one would like to understand the dependence of  $C(d, \delta, V,w,L)$
on the various parameters. For instance, one could hope, since the set $S_{\delta, L}$ is constructed using an equidistributed sequence $Z$, that the constant could be chosen uniformly in $L\in \NN^d$.
  \item The dependence of $C(d, \delta, V,w,L)$  on the radius $\delta>0$ is related to the question of vanishing order (of linear combinations) of eigenfunctions, cf.~\cite{DonnellyF-88,JerisonL-99}.
Now, \cite{Taeufer-1710.09328} shows that in the specific  situation that $V\equiv 0$, $L=(1,\ldots,1)$ the function $\delta \mapsto C(d, \delta, V,w,L)$  cannot be polynomial for any width $w>0$
and any dimension $d \geq 2$.
  \item While in \eqref{eq:UCP} $E$ represents the \emph{supremum} of the energy interval,
in \eqref{eq:question-E-independent-UCP} it plays the role of the position of the energy interval.
We are now asking for bounds independent of $E$, and if such exist, are interested in how the prefactor in \eqref{eq:question-E-independent-UCP} depends on the \emph{width} $w$ of the energy interval.
  \item
We mention now two different intuitions related to Question \ref{question}.
The first makes one inclined to expect a negative answer, the second to expect a positive one.

The volume of the energy shell in momentum space corresponding to $[E-w,E]$ grows polynomially with $E$.
Thus, the dimension of the family of linear subspaces of `test functions' where
the bound \eqref{eq:question-E-independent-UCP} is supposed to hold,  is not necessarily bounded, which could be an argument in favour of a negative answer to the question.

On the other hand, the parameter $E$ can be incorporated in the Hamiltonian since
$\chi_{(-\infty,E]}(H_L)=\chi_{[-\|V\|_\infty-E,0]}(H_L-E)$.
Such a shift is physically just a re-normalization and results merely in a phase factor for the unitary Schr\"odinger evolution:
$ \exp(-it(H_L-E))= \exp(itE) \exp(-itH_L)$
which is irrelevant for most physical questions of interest.
Thus, this dynamic point of view is an argument in favour of a positive answer to Question \ref{question}.
This intuition is supported by the fact that Theorem \ref{thm:small-energy-interval}
gives a partial positive answer to Question \ref{question} even in the case of non-vanishing potential
and its proof actually relies on dynamical methods.
\end{enumerate}

The rest of the paper is organized as follows:
In Section \ref{sec:results} we spell out four main theorems
and comment  them in Section \ref{sec:constants}.
In Section \ref{sec:analytic} we give tools on analytic functions related to the Turan Lemma, which are used in
Sections \ref{s:proofLS}, \ref{sec:proofKLS}, and \ref{sec:proof-T-2.3} to prove three of the main results.
The last theorem (concerning $V\not \equiv 0$) is proven in Section \ref{s:nontrivial-V}.

\section{Results on equidistribution properties of functions on tori}\label{sec:results}
Here we state our main theorems concerning equidistribution properties of functions in an appropriately chosen
linear subspace of $L^2(\TT_L^d)$.
We first consider the case that the subspace consists of
Fourier series on the torus with localized Fourier coefficients.
For this purpose we need some notation:
Let $d\geq 1$. For $L:=(L_1, \ldots, L_d)\in \RR^d_+$ we define
\[
\TT_L^d:=[0,2\pi L_1]\times\ldots\times[0,2\pi L_d],
\]
and for $k\in\ZZ^d$ we set $k/L:=(k_1/L_1,\ldots, k_d/L_d)$.

\begin{figure}[H]\centering
\begin{tikzpicture}           
\begin{scope}[xscale = 1.5, yscale = 1.5]
\filldraw[fill=gray!35] (1.5,2.7) rectangle (4.5,1.3);
\foreach \y in {0,1,2,3,4,5}{
 \draw (\y,0) -- (\y,3);
 }
\foreach \x in {0,0.5,1,1.5,2, 2.5,3}{
 \draw (0,\x) -- (5, \x);
 }
\foreach \x in {2,3,4}{
 \foreach \y in {1.5,2,2.5}{
   \filldraw[fill=gray!80] (\x,\y) circle (0.4mm);
   }
  }
\draw[<->, thick] (1.5,1.2) -- (4.5,1.2);
\draw (2.8,1.1) node {${b_1}$};
\draw[<->, thick] (1.35,1.3) -- (1.35,2.7);
\draw (1.2,1.8) node {${b_2}$};
\draw[<->, thick] (0,0.4) -- (1,0.4);
\draw (0.5,0.2) node {\small${ 1/L_1}$};
\draw[<->, thick] (0.2,0.5) -- (0.2,1);
\draw (0.5,0.75) node {\small${ 1/L_2}$};
\end{scope}
\end{tikzpicture}

\caption{$2$-dimensional example of the support of $\h{f}$.
The nodes denote the active Fourier modes as a subset of
$\big(\frac{1}{L_1}\ZZ\big)\times \big(\frac{1}{L_2}\ZZ\big)$
and the grey rectangle around them corresponds to the
$d$-dimensional rectangle $J$ as a subset of $\RR^2$.\label{fig:support}}
\end{figure}
For $f\in L^p(\TT_L^d)$ we adopt the convention
\begin{align}\label{eq:fourier-coefficients}
 \h{f}\colon \Gamma_{L} :=\left(\frac{1}{L_1}\ZZ\right)\times &\ldots\times \left(\frac{1}{L_d}\ZZ\right)\longrightarrow\RR\notag\\
\h{f}\left(\frac{k}{L}\right) = \h{f}\left(\frac{k_1}{L_1},\ldots,\frac{k_d}{L_d}\right) =
& \Big(\prod_{j=1}^d(2\pi L_j)\Big)^{-1}\int_{\TT_L^d} f(x)e^{-i x\cdot \frac{k}{L}}\dd x,
\end{align}
where $x\cdot (k/L):=\sum_{j=1}^d x_j k_j/L_j$.
In particular, $\supp\h{f}\subset \Gamma_{L}\subset\RR^d$ (see Figure \ref{fig:support}).

Our first two main results are as follows:

\begin{thm}\label{thm:1}
Let $f\in L^p(\TT_L^d)$ with $p\in[1,\infty]$ such that $\supp\widehat{f}\subset J$, where $J$ is a
 $d$-dimensional rectangle with sidelengths $b_1,\ldots, b_d$.
Let $S\subset \RR^d$ be a $(\gamma, a)$-thick set with $a=(a_1,\ldots,a_d)$ such that $0<a_j\leq 2\pi L_j$ for all $j=1,\ldots,d$.
Then,
\begin{equation}\label{eq:md-LogSer-torus-1.1}
\norm{f}{L^p(\TT_L^d)}\leq\Big(\frac{c^d}{\gamma}\Big)^{c a\cdot b+\frac{6d+1}{p}}\norm{f}{L^p(S\cap \TT_L^d)},
\end{equation}
where $c>0$ is a numerical constant and $a\cdot b$ stands for the Euclidean inner product in $\RR^d$.
\end{thm}

The statement of the theorem holds true whether $J$ intersects $\Gamma_L$ or not. (In the latter case the statement is trivial.)
\begin{thm}\label{thm:2}
Let $f\in L^p(\TT_L^d)$ with $p\in[1,\infty]$.
Let $n\in\NN$ and assume that $\supp\widehat{f}\subset \bigcup_{l=1}^n J_l$,
where $J_l$'s are $d$-dimensional rectangles, each with sidelengths $b_1,\ldots, b_d$.
 Set $b=(b_1,\ldots, b_d)$.
Let $S\subset\RR^d$ be $(\gamma, a)$-thick with $a=(a_1,\ldots,a_d)$ such that $0<a_j\leq 2\pi L_j$ for all $j=1,\ldots, d$. Then,
\begin{equation}\label{eq:md-LogSer-torus-2}
\norm{f}{L^p(\TT_L^d)}\leq\Big(\frac{\tilde{c}^d}{\gamma}\Big)^{\big(\frac{\tilde{c}^d}{\gamma}\big)^n a\cdot b+n-\frac{(p-1)}{p}}\norm{f}{L^p(S\cap\TT_L^d)},
\end{equation}
where $\tilde{c}\geq 3$ is a numerical constant.
\end{thm}

We emphasize that the estimates \eqref{eq:md-LogSer-torus-1.1} and \eqref{eq:md-LogSer-torus-2} are independent of the position
of the $d$-dimensional cubes $J$ and $J_l$, respectively, and of the number of active Fourier coefficients (or frequencies) contained in them.
Furthermore, they are \emph{scale-free}, in the sense that they are uniform in the scales $L_j$,
as long as $L_j\geq a_j/(2\pi)$ for all $j=1,\ldots, d$.

The prefactor in Theorem \ref{thm:1} is optimal up to the numerical constant $c$, as explained in
Section \ref{sec:constants}.
At the present point we do not have a clear understanding in which sense  the constant in Theorem \ref{thm:2} could be optimal (or not).
What is easy to see, is that
in the case of just one $d$-dimensional cube, i.e.~$n=1$, the estimate \eqref{eq:md-LogSer-torus-1.1} is better than estimate \eqref{eq:md-LogSer-torus-2},
since it shows a polynomial dependence on $1/\gamma$.
This is a manifestation of the different proofs used for the two theorems.

Note that since the function $f$ considered in Theorems \ref{thm:1} and \ref{thm:2} has a compactly supported Fourier transform it is in particular analytic.
Hence it is legitimate and natural to consider its extension to complex variables, which will be indeed done in Sections  \ref{sec:analytic} to \ref{sec:proofKLS} containing the proofs
of Theorems \ref{thm:1} and \ref{thm:2}.

Identifying $f$ with its periodic extension on $\RR^d$, Theorem \ref{thm:1} holds on every (shifted) periodicity cell.
Moreover, if such an inequality holds, the set $S$ must be thick.

\begin{theorem}\label{thm:sufficiency-necessity}
Let $S\subset\RR^d$ be a measurable set and let $p\in[1,\infty)$. The following statements are equivalent:
\begin{enumerate}
 \item [(i)] $S$ is thick;
 \item [(ii)] There exists an $L_0>0$ such that for all  $d$-dimensional rectangles $J$
 there exists a constant $C=C(S,J,d,p)$ such that for all $L\in\RR^d_+$ with
 $\min_{j=1,\ldots, d}L_j\geq L_0>0$, for all $f\in L^p(\TT_L^d)$ with $\supp\widehat{f} \subset J$,
 and for all $h\in\RR^d$ we have
 \begin{equation}\label{eq:extended-LS-torus-2}
 \norm{f}{L^p(\TT_L^d-h)} \leq C \norm{f}{L^p(S\cap(\TT_L^d-h))},
 \end{equation}
 where $f$ is identified with its periodic extension.

If $S$ is thick, the constant $C$ is as in Theorem \ref{thm:1}.
\end{enumerate}
\end{theorem}

\begin{remark}
\label{rmk:sufficiency-necessity}
We observe that the implication $(i)\Rightarrow (ii)$ holds also for $p=+\infty$.

Moreover, $(i)$ implies also
\begin{enumerate}
 \item [$(ii')$] There exists an $L_0>0$ such that for all $n \in \NN$, all
  $d$-dimensional rectangles $J_1, \ldots, J_n$
 there exists a constant $C$ depending only on $S,J_1, \ldots, J_n,d$, and $p$ such that for all $L\in\RR^d_+$ with
 $\min_{j=1,\ldots, d}L_j\geq L_0>0$, for all $f\in L^p(\TT_L^d)$ with $\supp\widehat{f} \subset J_1 \cup  \ldots \cup J_n$,
 and for all $h\in\RR^d$ we have
 \begin{equation}
 \norm{f}{L^p(\TT_L^d-h)} \leq C \norm{f}{L^p(S\cap(\TT_L^d-h))},
 \end{equation}
 where $f$ is identified with its periodic extension.
\end{enumerate}
\end{remark}

Note that $(ii')$ implies $(ii)$ and thus $(i)$.

Now we turn to the case that a potential $V\colon \TT^d \to \RR$ is added to the negative Laplacian and we are
studying functions in a spectral subspace of the Schr\"odinger operator $H:=-\Delta+V$.
Since the spectral decomposition of $H$ may be considered as a generalized type of Fourier transform this gives a kind of extension of
Theorem \ref{thm:1}. However, in this situation we are not able to make the prefactor in the inequality explicit.
Also, we consider here only the standard size torus $\TT^d:=[0,2\pi]^d$, although some of the results are valid also for general tori $\TT_L^d$.
We show that assembling the results of
\cite{AnantharamanM-14}, \cite{BurqZ-12}, \cite{BourgainBZ-13}, \cite{Bourgain-14}, \cite{RamdaniTTT-05}, and \cite{Miller-05c}
one arrives at the following

\begin{theorem}\label{thm:small-energy-interval}
Let $\emptyset \neq S\subset \TT^d$ be open and $V\in L^2(\TT^d)$.
Assume either that $d\in\{1,2,3\}$ or that $V\in L^\infty(\TT^d)$ and that its set of discontinuities  has measure zero.
Then there exist  constants $w=w(S,V)>0$, $\kappa=\kappa(S,V)>0$ such that
\begin{equation}\label{eq:small-energy-interval}
\forall \ f\in \bigcup_{E\in\RR} \Ran(\chi_{[E-w, E]}(H)): \quad \norm{f}{L^2(\TT^d)}\leq \kappa \norm{f}{L^2(S)}.
\end{equation}
\end{theorem}
The main point of Theorem \ref{thm:small-energy-interval} is that the prefactor $\kappa(S,V)$ is independent of the `energy shift' $E\in \RR$, which is the primary focus of Question \ref{question}.

One can extend the validity of \eqref{eq:small-energy-interval} even to
\begin{equation}\label{eq:shrinking-energy-interval}
\bigcup_{E\in\RR} \Ran(\chi_{[E-\widetilde w(E), E]}(H)): \quad \norm{f}{L^2(\TT^d)}\leq \kappa \norm{f}{L^2(S)}.
\end{equation}
where $\widetilde w\colon \RR \to (0, \infty)$ is a bounded continuous function satisfying
$\sup_{E\geq E_0}\widetilde w(E) \leq w(S,V)$ for some $E_0\in \RR$.
Namely, for $f\in \bigcup_{E\geq E_0} \Ran(\chi_{[E-\widetilde w(E), E]}(H))$
one can apply Theorem \ref{thm:small-energy-interval} and for
$f\in \bigcup_{E\leq E_0} \Ran(\chi_{[E-\widetilde w(E), E]}(H))$ Theorem \ref{thm:NTTV}.

\begin{remark}[Relevance for null-controllability of the heat equation]
Since the arXiv submission of the first version of this paper its results have
been used in several papers \cite{NakicTTV-18,EgidiV-18,NakicTTV-1810.10975,SeelmannV-1810.12167,EgidiS-1902.08141,GallaunST}
dealing with null-controllability of the heat equation.
For details we refer to the survey \cite{EgidiNSTTV-18} and the references therein.
\end{remark}

\section{Questions of optimality and discussion of the results}\label{sec:constants}

In this section we show that the constant in Theorem \ref{thm:1} is optimal,
that Theorem \ref{thm:2} identifies a strong annihilating pair,
that for $p=2$ our results recover Kovrijkine's theorem by passing to the limit $L\to\infty$,
and revisit Question \ref{question}.

We first demonstrate that the constant in Theorem \ref{thm:1} is generally optimal with respect to the parameters $\gamma$, $a$, and $b$,
 see also \cite[Examples 2.9-2.10]{EgidiNSTTV-18}.

\begin{example}\label{example:optimality-GenCase}
Let $a_1=\ldots =a_d=1$, $p\geq 1$, $b\geq 8\pi$, and $\epsilon\in (0,1)$.
We consider a periodized $\epsilon$-cube centred at the centre of each cell, more precisely the set
\[
S=F_1\times\ldots\times F_d\subset \RR^d
\]
such that each $F_j$ is 1-periodic and $F_j\cap[0,1]= \left[\frac{1}{2}-\frac{\epsilon}{2},\frac{1}{2}+\frac{\epsilon}{2} \right]$.
Then, $S$ is $(\gamma, 1)$-thick in $\RR^d$ with $\gamma=\epsilon^d$.

Let now $\NN\ni\alpha:=\left[\frac{b}{4\pi}\right] \geq 2$, $L=1/(2\pi)$, and $\tilde{L}=(L,\ldots,L)\in\RR^d$.
On the torus $\TT^1_{L}=[0,2\pi L]=[0,1]$ and on its $d$-dimensional counterpart  $\TT^d_{\tilde L}=[0,1]^d$ we consider the functions
\begin{align*}
& f:[0,1]\rightarrow \RR, \quad f(x)=(\sin(2\pi x))^\alpha,\\
& g: [0,1]^d \rightarrow \RR, \quad g(x)=\prod_{j=1}^d f(x_j)= \prod_{j=1}^d \sin(2\pi x_j)^\alpha.
\end{align*}
Clearly, $\supp\widehat{f}\subset [-2\pi\alpha, 2\pi\alpha]\subset \left[-\frac{b}{2},\frac{b}{2}\right]$,
$\supp \widehat{g}\subset\left[-\frac{b}{2},\frac{b}{2}\right]^d$, and the Fourier coefficients are equally spaced within their support.
Consequently, by Theorem \ref{thm:1} we have
\[
 \norm{g}{L^p(S\cap [0,1]^d)}\geq\Big(\frac{\epsilon^d}{c^d}\Big)^{c d b+\frac{6d+1}{p}}\norm{g}{L^p([0,1]^d)}.
\]
We now show that the pre-factor cannot be improved qualitatively, giving an upper bound on $\norm{g}{L^p(S\cap [0,1]^d)}$.
By separation of variables we have 
\[
\norm{g}{L^p(S\cap[0,1]^d)}= \prod_{j=1}^d \norm{f}{L^p(F_j\cap [0,1])}
\]
and similarly for $\norm{g}{L^p([0,1]^d)}$.
It is therefore enough to analyse the $L^p$-norm of $f$ on $F_1\cap [0,1]$.

By Jensen's inequality we have
\[
 \norm{f}{L^p([0,1])}^p=\int_0^1\abs{\sin(2\pi x)}^{p\alpha} \dd x \geq \left(\int_0^1 \abs{\sin(2\pi x)}\dd x\right)^{p\alpha}= \left(\frac{2}{\pi}\right)^{p\alpha}.
\]
By the change of variable $y=2\pi x-\pi$, the fact that $\abs{\sin(y+\pi)}=\abs{\sin(y)}$, 
the symmetry of the sine function, the inequality $\sin x \leq x$, and the
fact $\alpha +1/p \geq 1$, we estimate
\begin{align*}
\frac{ \norm{f}{L^p(F_1\cap [0,1])}}{ \norm{f}{L^p([0,1])}}
&
\leq \left(\frac{\pi}{2}\right)^{\alpha}\left(\int_{(1-\epsilon)/2}^{(1+\epsilon)/2} \abs{\sin(2\pi x)}^{p\alpha}\dd x\right)^{1/p}
\\
& = \left(\frac{\pi}{2}\right)^{\alpha}\left(\frac{1}{2\pi}\int_{-\pi\epsilon}^{\pi\epsilon} \abs{\sin(y)}^{p\alpha}\dd y\right)^{1/p}\\ 
& = \left(\frac{\pi}{2}\right)^{\alpha}\left(\frac{1}{\pi}\int_{0}^{\pi \epsilon} \sin^{p\alpha}(y)\dd y\right)^{1/p}
\leq \left(\frac{\pi}{2}\right)^{\alpha}\left(\frac{1}{\pi}\int_{0}^{\pi \epsilon} y^{p\alpha}\dd y\right)^{1/p}
\\
& = \left(\frac{\pi}{2}\right)^{\alpha}
\left(\frac{1}{\pi} \frac{(\pi\epsilon)^{1+p\alpha}}{1+p\alpha}\right)^{1/p}
 =
\frac{\epsilon^{1/p+\alpha} \pi^{2\alpha}}{2^\alpha (1+p\alpha)^{1/p}}
\quad  \leq \left(\frac{\epsilon}{(2/\pi^2)}\right)^{\alpha+1/p}.
\end{align*}
Since $\alpha +1/p =\left[\frac{b}{4\pi}\right] +1/p \geq \frac{b}{4\pi} -1$, for $\epsilon < 2/\pi^2$ we obtain
\begin{equation*}
 \norm{f}{L^p(F_1\cap [0,1])} \leq \left(\frac{\epsilon}{(2/\pi^2)}\right)^{\frac{b}{4\pi} -1}\norm{f}{L^p([0,1])},
\end{equation*}
which holds also for $\epsilon \geq  2/\pi^2$ trivially.
Consequently
\begin{equation*}
 \norm{g}{L^p(S\cap[0,1]^d)}
 \leq
 \left(\frac{\gamma}{(2/\pi^2)^d}\right)^{\frac{b}{4\pi} -1}\norm{g}{L^p([0,1]^d)}.
\end{equation*}
This shows that in general we cannot obtain a constant in \eqref{eq:md-LogSer-torus-1.1}
which is qualitatively better than $(c^d/\gamma)^{c(b+d)}$.
\end{example}

If we restrict to the subclass of functions whose Fourier coefficients are few but spread out, it is likely that there is a better bound,
as is illustrated by the following

\begin{example}\label{example:notOpt-subclass}
Let $b\in \NN$, $\gamma\in (0,1)$, $S$ be the $1$-periodic set such that $S\cap [0,1]=[0,\gamma]$, and
$f\colon [0,1]\to \RR$ be the function $f(x)=\sin(2b\pi x)$.
This function has two non-zero Fourier coefficients at $-2b\pi$ and $2b\pi$,
growing further apart as $b$ increases. For the $L^1$-norm of $f$ on $[0,1]$ and $[0,\gamma]$ we calculate
\begin{equation*}
 \frac{\norm{f}{L^1([0,\gamma])}}{\norm{f}{L^1([0,1])}} \leq \frac{\pi}{2}\int_0^\gamma 2b\pi x \dd x = \frac{\pi^2}{2} b\gamma^2,
\end{equation*}
suggesting a behaviour of type $b\gamma^2$ instead of $\gamma^b$, as in Theorem \ref{thm:1}.
\end{example}

\begin{remark}\label{rmk:annihilating-pair}
 The definition of annihilating pairs in \eqref{eq:annihilating-pair} can be adapted for functions in $L^2(\TT_L^d)$,
 see \cite[p. 102]{HavinJ-94} for a formulation in the case $d=1$ and $L=1$.
 Let $\mathcal{E}\subset \TT_L^d$ be measurable and $\mathcal{B}\subset\Gamma_{L}:= (\frac{1}{L_1}\ZZ)\times\ldots\times(\frac{1}{L_d}\ZZ)$.
 We say that $(\mathcal{E},\mathcal{B})$ is a strong annihilating pair
 if there exists a constant $C_a=C_a(\mathcal{E}, \mathcal{B})$ such that
 \begin{align}
   \label{eq:discrete-annihilating-pair}
  \forall\ f\in L^2(\TT_L^d): \quad  \norm{f}{L^2(\TT_L^d)}^2
  \leq
  C_a\Big(\norm{f}{L^2(\mathcal{E}^c)}^2  +\norm{\widehat{f}\chi_{\mathcal{B}^c}}{\ell^2(\Gamma_{L})}^2 \Big)
  \\
  \intertext{where} \nonumber
  \quad \norm{\widehat{f}\chi_{\mathcal{B}^c}}{\ell^2(\Gamma_{L})}^2
  =(2\pi)^d (L_1\cdots L_d) \sum_{k/L\in\mathcal{B}^c} \abs{\widehat{f}(k/L)}^2 .
 \end{align}
 For $p=2$,  Theorem \ref{thm:2} implies that the pair $(\TT_L^d\setminus S, \mathcal{G})$ is strong annihilating,
 where   $\mathcal{G}= \Gamma_{L} \cap \bigcup_{j=1}^n J_n$. Indeed, let $f\in L^2(\TT_L^d)$ and denote by $\mathcal{F}^{-1}$ the Fourier inversion.
 Then, by Parseval Theorem and Pythagoras we obtain
 \begin{align*}
  \norm{f}{L^2(\TT_L^d)}^2 =\norm{\widehat{f}}{\ell^2(\Gamma_{L})}^2
  & = \norm{\widehat{f}\chi_\mathcal{G}}{\ell^2(\Gamma_{L})}^2 +\norm{\widehat{f}\chi_{\mathcal{G}^c}}{\ell^2(\Gamma_{L})}^2 \\
  & = \norm{\mathcal{F}^{-1}(\widehat{f}\chi_\mathcal{G})}{L^2(\TT_L^d)}^2 +\norm{\widehat{f} }{\ell^2(\mathcal{G}^c)}^2
 \end{align*}
 Denoting the prefactor in Theorem  \ref{thm:2} by $C$ and using the triangle inequality for $\widehat{f}\chi_{\mathcal{G}}=\widehat{f}- \widehat{f}\chi_{\mathcal{G}^c}$ gives
 \begin{align*} \\
  \norm{f}{L^2(\TT_L^d)}^2
  & \leq 2 C^2 \norm{\mathcal{F}^{-1}(\widehat{f})}{L^2(S\cap \TT_L^d)}^2 +2 C^2 \norm{\mathcal{F}^{-1}(\widehat{f}\chi_{\mathcal{G}^c})}{L^2(S\cap \TT_L^d)}^2
  +\norm{\widehat{f} }{\ell^2(\mathcal{G}^c)}^2\\
  & \leq  2 C^2  \norm{f}{L^2(S\cap \TT_L^d)}^2 +2 C^2 \norm{\mathcal{F}^{-1}(\widehat{f}\chi_{\mathcal{G}^c})}{L^2(\TT_L^d)}^2 +\norm{\widehat{f}}{\ell^2(\mathcal{G}^c)}^2
  \\
  & \leq  (1+2 C^2 )\left( \norm{f}{L^2(S\cap \TT_L^d)}^2 +\norm{\widehat{f}}{\ell^2(\mathcal{G}^c)}^2 \right)
  \end{align*}
which proves  \eqref{eq:discrete-annihilating-pair} with $C_a=(1+2 C^2 )$.
Note that this constant is again scale-free.
\end{remark}

\begin{remark}\label{rmk:approximation}
For $p=2$, Theorem \ref{thm:kovrijkine} can be recovered from Theorems \ref{thm:1} and \ref{thm:2}.
For notational simplicity let us first consider the case $d=1$.
Let $J$ be an interval of length $b$ and let $f\in L^2(\RR)$ with $\supp \widehat{f}\subset J$.
W.l.o.g. we may assume that $\widehat{f}$ is continuous. Indeed, if this is not the case we approximate $\widehat{f}$ in $L^2(\RR)$ by a sequence $(\eta_j)_{j\in\NN}$
of continuous functions with support in $J$. By Plancherel Theorem \cite[Corollary 4.7]{MuscaloS-13} $f$ is approximated in $L^2(\RR)$
by the sequence obtained from $(\eta_j)_{j\in\NN}$ by Fourier inversion.

For $L\in\NN$, we define the functions
 \[
  g_L(x):=\left\{\begin{array}{ll}
          \sum\limits_{{\frac{k}{L}\in (J\cap\frac{1}{L}\ZZ)}}\frac{1}{L} \widehat{f}\big(\frac{k}{L}\big)e^{i\frac{k}{L} x} & \quad x\in [0,2\pi L]=\TT_L,\\
          0 & \quad x\not\in\TT_L,
         \end{array}\right.
 \]
so that $(g_L)_{L\in\NN}\subset L^2(\RR)$, and $\supp\widehat{g_L}\subset J$ when considered as $L^2(\TT_L)$-functions.
Since the $g_L$'s are constructed using a Riemann sum, by means of Fourier inversion we obtain
\[
\lim_{L\to+\infty} g_L(x) = \int_J \widehat{f}(\xi)e^{i\xi x}\dd \xi = f(x)\quad \text{a. e.}
\]
Moreover, Parseval and Plancherel Theorems \cite[Corollary 1.6 and Corollary 4.7]{MuscaloS-13} together with the use of Riemann sums yield
 \[
 \lim_{L\to+\infty}\norm{g_L}{L^2(\RR)} =\norm{f}{L^2(\RR)}.
 \]
 Since pointwise convergence and convergence of the $L^2$-norm
 imply weak convergence in $L^2(\RR)$ \cite[Proposition 4.7.12]{Bogachev-07}, and weak convergence in $L^2(\RR)$ and convergence of the $L^2$-norm
 imply strong convergence in $L^2(\RR)$ \cite[Corollary 4.7.16]{Bogachev-07},
 we conclude that for all measurable $\Omega\subset\RR$ we have
 \[
 \lim_{L\to+\infty}\norm{g_L-f}{L^2(\Omega)}=0.
 \]
 Let now $S$ be $(\gamma, a)$-thick in $\RR$. Since $a<\infty$, there exists $L_0>0$ such that $a\leq 2\pi L_0$. For all $L\geq L_0$, Theorem \ref{thm:1} yields
 \[
  \norm{g_L}{L^2(\RR)}=\norm{g_L}{L^2(\TT_L)}
  \leq \left(\frac{c}{\gamma}\right)^{cab+\frac{7}{2}} \norm{g_L}{L^2(S\cap\TT_L)}
  =    \left(\frac{c}{\gamma}\right)^{cab+\frac{7}{2}}\norm{g_L}{L^2(S)},
 \]
 and taking the limit $L\to+\infty$ we obtain
 \[
  \norm{f}{L^2(\RR)}\leq \left(\frac{c}{\gamma}\right)^{cab+\frac{7}{2}}\norm{f}{L^2(S)},
 \]
 that is, the statement of Theorem \ref{thm:kovrijkine}~(I) for $p=2$ and $d=1$.

 Similarly, considering a function $f\in L^2(\RR)$ with $\supp\widehat{f}\subset \bigcup_{l=1}^n J_l$ we are able to recover
 Theorem \ref{thm:kovrijkine}~(II).  For general dimensions $d$ the notation is more cumbersome, but the argument is completely analogous.
\end{remark}

Consider $p=2$ and the case that $\TT_L^d$ is a $d$-dimensional cube,
i.e. $\TT_L^d=[0,2\pi L]^d$. In this case, Theorem \ref{thm:1} gives us an equidistribution property of
linear combinations of eigenfunctions of the operator $-\Delta_L^P$, where $\Delta_L^P$ is the Laplacian on $\TT_L^d$ with periodic boundary conditions.
Analogously as in  \eqref{eq:spectral-inequality}, let $E>0$ and let $f\in \Ran(\chi_{(-\infty,E]}(-\Delta_L^P))$, then  $\supp\widehat{f}\subset[-\sqrt{E}, \sqrt{E}]^d$
and Theorem \ref{thm:1} turns into the estimate
\begin{equation}\label{eq:spectral-inequality-torus}
 \norm{f}{L^2(\TT_L^d)}^2
 \leq 
 \left(\frac{c^d}{\gamma}\right)^{2c\abs{a}_1\sqrt{E}+\frac{7}{2}d}
 \norm{f}{L^2(S\cap\TT_L^d)}^2,
 \quad \text{ where } \quad \abs{a}_1:=\sum_{j=1}^d a_j.
\end{equation}
This inequality motivates us to revisit Question \ref{question} in the case $V\equiv 0$ in the following
\begin{example}\label{example:energy-shell}
Consider a function $f$ with $\supp\widehat{f}\subset\Sigma\subset\RR^d$ where
$\Sigma=\{p\in\RR^d\;\vert\; (E-1)\leq\norm{p}{2}^2\leq E\}$ for a given $2<E\in\RR$.
We want to cover $\Sigma$ with disjoint $d$-dimensional cubes, in order to use Theorem \ref{thm:2}.
When $d=1$, Taylor expansion shows that $n=2$ intervals of length $1/\sqrt{2E}$ suffice for the covering and the prefactor in Theorem \ref{thm:2} reduces to
\[
\left(\frac{c}{\gamma}\right)^{\left(\frac{c}{\gamma}\right)^{2}a/\sqrt{E}+2-\frac{1}{2}}
\leq \left(\frac{c}{\gamma}\right)^{\left(\frac{c}{\gamma}\right)^{2}a+2-\frac{1}{2}}.
\]
If $d=2$ the volume of the spherical shell  $B(0,\sqrt{E})\setminus B(0,\sqrt{E-1})$ is bounded by  $\pi$
and is in particular independent of $E$. However, if one wants to cover the shell with unit $d$-dimensional cubes,
their number must grow with $E$ due to the unbounded growth of the circumference of the annulus.
Thus  Theorem \ref{thm:2} does not lead to a positive answer to Question \ref{question} in $d\geq 2$ even if $V\equiv 0$.

Note that, if we replace $\Sigma$ by $\Sigma_w=\{p\in\RR^d\;\vert\; (E-w)\leq\norm{p}{2}^2\leq E\}$
for a sufficiently small $w$ (depending on the open observability set $S\neq \emptyset$ under consideration)
Theorem \ref{thm:small-energy-interval}  is applicable.
\end{example}

\section{Analytic Tools motivated by the Turan lemma}\label{sec:analytic}

The following two results are inspired by \cite{Nazarov-94} and their proofs can be found in \cite{Kovrijkine-thesis,Kovrijkine-01}.
The first one is a key ingredient for the proof of Theorem \ref{thm:1}, while the second plays an analogous role for Theorem \ref{thm:2}.
In particular, the proofs of Theorem \ref{thm:1} and of Theorem \ref{thm:2} will consist in reducing the situation to
the following theorems (and Lemma \ref{lemma:level-set-argument}).

The following result is formulated in Lemma~1 (and its proof) of \cite{Kovrijkine-thesis} and \cite{Kovrijkine-01} where we refer for a proof.
It relates the sup norm of an analytic function on a real interval to its sup norm on a positive measure subset.
As an a-priori information one has to know a bound on the sup norm of the function on a sufficiently large complex disc surrounding the real interval.

\begin{thm}\label{lemma:2}
Let $z_0\in\RR$ and let $\phi$ be an analytic function on $D(z_0,5):= \{z \in \CC \mid |z-z_0| < 5\}$.
We consider $I\subset\RR$ an interval of unit length such that $z_0\in I$ and $A\subset I$ a measurable set of non-zero measure, i.e., $\abs{A}>0$.
Set $M=\max_{\abs{z-z_0}\leq 4}\abs{\phi(z)}$ and assume that $\abs{\phi(z_0)}\geq 1$, then 
\begin{equation}\label{eq:lemma1}
\sup_{x\in I}\abs{\phi(x)}\leq\Bigg(\frac{12}{\abs{A}}\Bigg)^{2\frac{\log M}{\log 2}}\cdot\sup_{x\in A}\abs{\phi(x)}.
\end{equation}

\end{thm}

\begin{remark}
\begin{enumerate}
\item
We have $\log M /\log 2 \geq 1$ unless $\phi$ has no zeros in the disc $D(z_0,2)\subset \CC$. 
In the latter case we have $\sup_{x\in I}\abs{\phi(x)}\leq M^3 \cdot\sup_{x\in A}\abs{\phi(x)}$.
This follows from the standard Harnack inequality for positive harmonic functions.
\item
If $\phi$ is non-constant the assumptions of the lemma ensure $\log M>\log \abs{\phi(z_0)}\geq 0$.
 \item 
In the Lemma 1 in \cite{Kovrijkine-01} the exponent is $\frac{\log M}{\log 2}$ rather than $2\frac{\log M}{\log 2}$ as stated in Inequality \eqref{eq:lemma1}.
We were only able to reproduce the proof with this extra factor of two.  For the application in the rest of the paper this factor is irrelevant, so we are conservative and include it.
\end{enumerate}

\end{remark}

The following theorem is an extension of a classical lemma by Turan \cite{Turan-84}.

\begin{thm}\label{lemma:3}
If $r(x)=\sum_{k=1}^n p_k(x)e^{i\lambda_k x}$, where $p_k(x)$ are polynomials of degree at most $m-1$ and $\lambda_1,\ldots, \lambda_n \in\RR$,
and $A$ is a measurable subset of an interval $I$ such that $\abs{A}>0$, then
\begin{equation}\label{eq:lemma3}
\norm{r}{L^\infty(I)}\leq\Big(\frac{316\abs{I}}{\abs{A}}\Big)^{nm-1}\norm{r}{L^\infty(A)}.
\end{equation}
\end{thm}

For constant coefficients $p_k$ this was proven by F.~Nazarov  in \cite[Theorem 1.5]{Nazarov-94},
while Kovrijkine adapted the estimate to polynomial coefficients, see Lemma~3 in \cite{Kovrijkine-thesis} and \cite{Kovrijkine-01}.
We do not give a proof but refer the reader to \cite{Nazarov-94} and \cite{Kovrijkine-thesis,Kovrijkine-01}.

In contrast to Theorem \ref{lemma:2}, Theorem \ref{lemma:3} considers a more restrictive class of functions.
On the other hand, Theorem \ref{lemma:3} does not need the sup norm of the exponential polynomial on a complex disc as a-priori information, but only its degree.

The following lemma will be recalled both in the proof of Theorem \ref{thm:1} and of Theorem \ref{thm:2}.
It casts in a quantitative form the idea that a non-negative function $|f|$ can be substantially smaller than its average only on a set of small measure.

\begin{lemma}\label{lemma:level-set-argument}
Let $U\subset\Lambda\subset\RR^d$ be  measurable sets with $\abs{\Lambda}=1$ and $\abs{U}>0$.
Let $p\in [1,\infty]$, $f\in L^p(\Lambda)$,  $C\in [1,\infty)$, $\alpha\in (0,\infty)$,
and  $Q\in [0,\infty)$.
For $\epsilon=\frac{C}{1+C}\abs{U}$ let
\begin{equation} \label{eq:level-set-1}
{W}=\{x\in \Lambda\;\vert\;\abs{f(x)}+Q<\left(\frac{\epsilon}{C}\right)^\alpha\norm{f}{L^p(\Lambda)}\}.
\end{equation}
Assume that
\begin{equation}\label{eq:assumption}
\sup_{x\in U}\abs{f(x)}+Q\geq\left(\frac{\abs{U}}{C}\right)^\alpha\norm{f}{L^p(\Lambda)},
\end{equation}
and
\begin{equation}\label{eq:assumption1}
\sup_{x\in {W}}\abs{f(x)}+Q\geq\left(\frac{\abs{{W}}}{C}\right)^\alpha\norm{f}{L^p(\Lambda)}.
\end{equation}
Then, $\abs{{W}}\leq \epsilon$ and
\begin{equation}\label{eq:level-set-2}
\norm{\abs{f}+Q}{L^{q}(U)}^q\geq\left(\frac{\abs{U}}{1+C}\right)^{q\alpha+1}\norm{f}{L^{q}(\Lambda)}^q,
\end{equation}
for any $1\leq q\leq p$ and $q<\infty$.
\end{lemma}
\begin{proof}
We first observe that if $\abs{{W}}=0$, then $\abs{{W}}\leq\epsilon$ is trivially satisfied. Hence, let $\abs{{W}}>0 $.
Similarly, if $\sup_{x\in {W}}\abs{f(x)}+Q=0$ then \eqref{eq:assumption1} implies $\norm{f}{L^p(\Lambda)}=0$ and \eqref{eq:level-set-2} holds trivially.
Thus, assume  $\sup_{x\in {W}}\abs{f(x)}+Q>0$.
Then, by the definition and assumption on ${W}$ we have
\begin{equation*}\label{eq:eq:32}
\begin{split}
\sup_{x\in {W}}\abs{f(x)}+Q
\leq\left(\frac{\epsilon}{C}\right)^\alpha\norm{f}{L^p(\Lambda)}
&\leq \left(\frac{\epsilon}{C}\right)^\alpha\left(\frac{C}{\abs{{W}}}\right)^\alpha\left(\sup_{x\in {W}}\abs{f(x)}+Q\right)\\
&\leq\left(\frac{\epsilon}{\abs{{W}}}\right)^\alpha\left(\sup_{x\in {W}}\abs{f(x)}+Q\right),
\end{split}
\end{equation*}
and it follows that $\frac{\epsilon}{\abs{{W}}}\geq 1$.

Let ${U\setminus W}=\{x\in U\;\vert\;\abs{f(x)}+Q\geq\left(\frac{\epsilon}{C}\right)^\alpha\norm{f}{L^p(\Lambda)}\}$
be the complement of ${W}$ in $U$. Then $\abs{{U\setminus W}}\geq\abs{U}-\epsilon \geq(1+C)^{-1}\abs{U}$,
since  $\epsilon=\frac{C}{1+C}\abs{U}$.
Now
\begin{equation*}
\begin{split}
\int_{U}(\abs{f(x)}+Q)^{q}
&\geq\int_{{U\setminus W}}(\abs{f(x)}+Q)^{q}
\geq\left(\frac{\abs{U}}{1+C}\right) \left(\frac{\abs{U}}{1+C}\right)^{q\alpha}\norm{f}{L^p(\Lambda)}^{q}\\
&\geq\left(\frac{\abs{U}}{1+C}\right)^{q\alpha+1}\norm{f}{L^{q}(\Lambda)}^{q},
\end{split}
\end{equation*}
where we used H\"older's inequality  $\norm{f}{L^q(\Lambda)} \leq  \abs{\Lambda}^{1/q-1/p} \norm{f}{L^p(\Lambda)}=\norm{f}{L^p(\Lambda)}$,
since $\Lambda$  has measure one.
\end{proof}

We now recall the Bernstein Inequality for periodic functions (see \cite[Proposition 1.11]{MuscaloS-13} and \cite[Theorem 11.3.3]{Boas-54}),
which will be used in both proofs.

\begin{proposition}[Bernstein's Inequality]\label{prop:bernstein}
Let $d\geq 1$, $p\in [1,\infty]$, and $f\in L^p(\TT_L^d)$ such that $\supp\h{f}\subset [-b_1,b_1]\times\ldots\times[-b_d,b_d]$.
Set $b=(b_1,\ldots,b_d)$, then
\begin{equation}\label{eq:bernstein}
\norm{\partial^{\alpha}f}{L^p(\TT_L^d)}\leq C_B^{\abs{\alpha}}b^{\alpha}\norm{f}{L^p(\TT_L^d)}
\quad \forall \alpha=(\alpha_1, \ldots, \alpha_d)\in\NN_0^d.
\end{equation}
Here $1\leq C_B\in\RR$ is a numerical constant.
\end{proposition}

The value of the numerical constant $C_B$ depends on the technique used in the proof. 
In \cite[Theorem 11.3.3]{Boas-54} the author writes $\partial^\alpha f$ as a convex combination
of particular values of the function itself and uses the convexity of the function $t\mapsto t^p$, 
$p\geq 1$, obtaining $C_B=1$, while in \cite[Proposition 1.11]{MuscaloS-13} the authors use 
a convolution method and apply Young's Inequality, leading to a value of $C_B$ much larger than one.

The following lemma is used in the proof of Theorem \ref{thm:sufficiency-necessity} and can be again considered as a Bernstein-type bound.

\begin{lemma}\label{lem:L-infinity-L-p-estimate}
Let $p\in[1,\infty]$ with dual exponent $q$,  $f\in L^p(\TT_L^d)$ such that $\supp\widehat{f}\subset J$, where $J$ is a
 $d$-dimensional rectangle with sidelengths $b_1,\ldots, b_d$.
Then,
\begin{equation}\label{eq:L-infinity-L-p-estimate}
 \norm{f}{L^\infty(\TT_L^d)}
 \leq
 (2\pi)^{d/q} \prod_{j=1}^d \left(\frac{2L_j b_j +1}{L_j^{1/p}}\right)   \norm{f}{L^p(\TT_L^d)},
\end{equation}
\end{lemma}

\begin{proof}
Without loss of generality, we may assume that $J$ is centred at zero. Indeed, if this is not the case, we shift $J$ by
$c\in\left(\left(\frac{1}{L_1}\ZZ\right)\times\ldots\times\left(\frac{1}{L_d}\ZZ\right)\right)\cap J$ so that
$J-c\subset \tilde{J}:=[-b_1,b_1]\times\ldots\times[-b_d,b_d]$. This affects $f$ only by multiplication with the factor $e^{i c\cdot x}$
and does not change its $L^p$-norm.

We first treat the case $L=(1,\ldots,1)$. Let $k\mapsto \chi_{\tilde{J}\cap \ZZ^d}\left(k\right)$ be the characteristic function of $\tilde{J}\cap \ZZ^d$,
and let $h:\TT^d\rightarrow\CC$ such that $\widehat{h}=\chi_{\tilde{J}\cap \ZZ^d}$.
Then $h\in L^q(\TT^d)$ for $1=\frac{1}{q}+\frac{1}{p}$. The triangle inequality gives
\[
\norm{h}{L^q(\TT^d)}
\leq
\sum_{k\in \tilde{J}\cap \ZZ^d} \norm{ \mathcal{F}^{-1} \chi_{\{k\}} }{L^q(\TT^d)}
\leq
(2\pi)^{d/q} \left(\prod_{j=1}^d(2b_j+1)\right)
\]
where we denote by $\mathcal{F}^{-1}$ the inverse Fourier transform.
Since $\widehat{f} = \widehat{h} \cdot \widehat{f}$, by Fourier inversion we obtain $f= h* f$ and by Young's Inequality \cite[Lemma 1.1~(ii)]{MuscaloS-13}
\begin{equation}\label{eq:estimate-1}
 \norm{f}{L^\infty(\TT^d)}
 \leq
 \norm{h}{L^q(\TT^d)} \norm{f}{L^p(\TT^d)}
 =
 (2\pi)^{d/q}  \left(\prod_{j=1}^d(2b_j+1)\right)\norm{f}{L^p(\TT^d)}.
\end{equation}
Let now $L\in\RR^d_+$ and $f\in L^p(\TT_L^d)$ with $\supp\widehat{f}\subset \tilde{J}$.
Scaling by $L_j$ in each coordinate, we obtain the function $g(x)=f(L_1 x_1,\ldots, L_d x_d)$, $x\in\TT^d$, for which
we have $\supp\widehat{g}\subset [-L_1 b_1,L_1 b_1]\times\ldots\times[-L_d b_d,L_d b_d]$,
$\norm{g}{L^\infty(\TT^d)}=\norm{f}{L^\infty(\TT_L^d)}$,
and $\norm{g}{L^p(\TT^d)}=\left(\prod_{j=1}^d L_j\right)^{-1/p} \norm{f}{L^p(\TT_L^d)}$. Then, by \eqref{eq:estimate-1}
\begin{align*}
 \norm{f}{L^\infty(\TT_L^d)} & =
 \norm{g}{L^\infty(\TT^d)} \\
 &\leq
  (2\pi)^{d/q} \prod_{j=1}^d \left(2L_j b_j +1\right) \norm{g}{L^1(\TT^d)}
 =
  (2\pi)^{d/q} \prod_{j=1}^d \frac{2L_j b_j +1}{L_j^{1/p}}  \norm{f}{L^p(\TT_L^d)},
\end{align*}
which proves the claim.
\end{proof}

Finally, we present a statement inspired by a claim in the proof of \cite[Theorem 2']{Kovrijkine-01}.
In what follows
we consider  $d$-dimensional rectangles $J_1,\ldots,J_n$ of the form
\begin{equation}\label{eq:parallelepipeds}
J_l =\left[\lambda_{l,1}- \frac{b_1}{2},\lambda_{l,1}+ \frac{b_1}{2}\right] \times\ldots\times
\left [\lambda_{l,d}- \frac{b_d}{2},\lambda_{l,d}+ \frac{b_d}{2}\right] \quad
\end{equation}
for all $l= 1, \ldots, n$. The $d$-dimensional rectangles $J_1,\ldots,J_n$ will be called \emph{disjoint}, if
\[
\forall \, l,k\in\{1,\ldots,n\}, l\neq k, \, \exists  \, j\in\{1,\ldots,d\}: \ |\lambda_{l,j}-\lambda_{k,j}|>2b_j
\]
\begin{lemma}\label{lemma:sum-estimate}
Let the  $d$-dimensional rectangles $J_1,\ldots,J_n\subset \RR^d $ be disjoint.
Let $f\in L^p(\TT_L^d)$ for $p\in[1,\infty]$ be such that $\supp\h{f}\subset J_1\cup\ldots\cup J_n$
 Then,
$f(x)=\sum_{l=1}^n f_l(x)e^{i c_l\cdot x}$,
where each $f_l$ satisfies $\supp\h{f_l}\subset  J_l-c_l\subset [-b_1,b_1]\times\ldots\times [-b_d,b_d]$, with
$c_l\in \left(\frac{1}{L_1}\ZZ\times\ldots\times\frac{1}{L_d}\ZZ\right)\cap J_l$, and
\begin{equation}\label{eq:sum-estimate}
\norm{f_l}{L^p(\TT_L^d)}\leq K^d\norm{f}{L^p(\TT_L^d)},
\end{equation}
where $K$ is a numerical constant.
It can be chosen equal to $6\pi$.
\end{lemma}

\begin{proof}
The first part of the lemma follows by properties of the Fourier series. In fact, we perform a shift (in Fourier space) of each
 $d$-dimensional rectangle
$J_l$ by a point $c_l\in\left(\frac{1}{L_1}\ZZ\times\ldots\times\frac{1}{L_d}\ZZ\right)\cap J_l$ so that $J_l-c_l\subset [-b_1,b_1]\times\ldots,\times[-b_d,b_d]$.
Consequently, we obtain $\h{f}(\frac{1}{L}k)=\sum_{l=1}^n\h{f}_l(\frac{1}{L}k-c_l)$,
where $f_l$'s are functions whose Fourier coefficients are contained in $[-b_1,b_1]\times\ldots\times[-b_d,b_d]$.
Thus, $f(x)=\sum_{l=1}^n f_l(x)e^{i c_l\cdot x}$.

To prove the second part, we first assume that $L=(1,\ldots, 1)$. 
Let $h:\TT^d\rightarrow\CC$ such that $\h{h}(k)=\chi_{\{-1,0,1\}^d}(k)$.
It is easy to check that $h\in L^1$ and $\norm{h}{L^1(\TT^d)}\leq (6\pi)^d$.
We also set $\phi(x):= h(b_1x_1,\ldots,b_dx_d)$, defined on $[0,2\pi/b_1]\times\ldots\times [0,2\pi/b_d]$ and such that $\h{\phi}(b_1k_1,\ldots,k_db_d)=\h{h}(k)$.
Consequently $\supp\h{\phi}\subset [-b_1,b_1]\times\ldots\times [-b_d,b_d]$. Then,
\begin{equation}\label{eq:eq:7}
\h{f_l}(k)=\h{f}(k+c_l)\h{\phi}(b_1k_1,\ldots,b_dk_d)=\reallywidehat{e^{-ic_l\cdot x}f\ast \phi(\cdot/b)}(k)=\reallywidehat{e^{-ic_l\cdot x}f\ast h}(k),
\end{equation}
and by Young's Inequality \cite[Lemma 1.1~(ii)]{MuscaloS-13} we conclude
\begin{equation}\label{eq:L=1}
\norm{f_l}{L^p(\TT^d)}\leq\norm{h}{L^1(\TT^d)}\norm{f}{L^p(\TT^d)}\leq (6\pi)^d \norm{f}{L^p(\TT^d)}.
\end{equation}

If $L\neq (1,\ldots,1)$, let $T:\TT^d\longrightarrow\TT_L^d$, $T(x)=(L_1 x_1,\ldots , L_d x_d)$
to obtain the function
\begin{equation}\label{eq:eq:29}
g(x)=(f\circ T)(x)=\sum_{l=1}^n (f_l\circ T)(x)e^{i c_l\cdot T(x)}:=\sum_{l=1}^n g_l(x)e^{i c_l\cdot T(x)}.
\end{equation}
Let first $p\in[1,\infty)$. For the $L^p$-norms of $g$ and $g_l$ we have
\begin{equation}\label{eq:eq:8}
\norm{f}{L^p(\TT_L^d)}=\Big(\prod_{j=1}^d L_j\Big)^{-1/p}\norm{g}{L^p(\TT^d)}
\,\text{ and }\,
\norm{f_l}{L^p(\TT_L^d)}=\Big(\prod_{j=1}^d L_j\Big)^{-1/p}\norm{g_l}{L^p(\TT^d)}.
\end{equation}
Then, using \eqref{eq:L=1} we obtain
\begin{equation}\label{eq:eq:9}
 \begin{split}
\norm{f_l}{L^p(\TT_L^d)}& =\Big(\prod_{j=1}^d L_j\Big)^{-1/p}\norm{g_l}{L^p(\TT^d)}\\
&\leq (6\pi)^d \Big(\prod_{j=1}^d L_j\Big)^{-1/p}\norm{g}{L^p(\TT^d)}
\leq (6\pi)^d \norm{f}{L^p(\TT_L^d)}.
\end{split}
\end{equation}
For $p=\infty$, we have $\norm{f_l}{L^\infty(\TT_L^d)}=\norm{g_l}{L^\infty(\TT^d)}$
and $\norm{f}{L^\infty(\TT_L^d)}=\norm{g}{L^\infty(\TT^d)}$.
Then the claim follows similarly as in \eqref{eq:eq:9}, and the proof is concluded.
\end{proof}

\section{Proof of Theorem \ref{thm:1}}\label{s:proofLS}

The strategy of the proof consists of two parts.
First one splits $\TT_L^d$ into unit cubes. (Here and in the following unit cubes will always mean \emph{unit $d$-dimensional cubes}.)
This is necessary since the diameter of $\TT_L^d$ depends on $L$
while the one of the unit cube is simply $\sqrt{d}$, enabling  an application of Theorem \ref{lemma:2} \emph{without}
$L$-scaling. Furthermore, the Lebesgue measure on the unit cube is normalized, which  allows the use of a probabilistic trick.
In fact, the idea used several times is that an estimate on the integral over the unit cube, i.\,e.~the average, of a function $ \rho\geq 0$
implies a point-wise estimate for the function values, on a set which is not too small with respect to the induced measure with density $\rho$.
We have encountered a variant of this idea in Lemma \ref{lemma:level-set-argument}.

In the second part one identifies a sufficiently rich class of cubes,
such that on each one of them it is possible to estimate the maximum of $\vert f \vert$.

We assume that $J$ is centred at $\lambda=(\lambda_1,\ldots,\lambda_d)$, i.e.
\begin{equation}\label{eq:eq:30}
J=[\lambda_1-b_1/2,\lambda_1+b_1/2]\times\ldots\times[\lambda_d-b_d/2,\lambda_d+b_d/2],
\end{equation}
and we consider $c\in\left(\frac{1}{L_1}\ZZ\times\ldots\times\frac{1}{L_d}\ZZ\right)\cap J$ so that $J-c\subset [-b_1,b_1]\times\ldots\times[b_d,b_d]$.
This shift (in Fourier space) affects $f$ only by multiplication with the factor $e^{i c\cdot x}$
and does not change its $L^p$-norm. Hence, without loss of generality we can assume that $\widehat{f}$ has support in $[-b_1,b_1]\times\ldots\times[-b_d,b_d]$.
Moreover, we first assume $p\in[1,\infty)$, $a=(1, \ldots, 1)$ and $2\pi L_j\geq 1$ for all $j=1,\ldots,d$.
We will then recover the general estimate by a scaling argument in $a$.

Let $\hat\Gamma=\ZZ^d\cap ([0,\lfloor 2\pi L_1\rfloor -1]\times\ldots\times[0,\lfloor 2\pi L_d\rfloor -1])$ and 
$\tilde\Gamma=\ZZ^d\cap ([0,\lceil 2\pi L_1\rceil -1]\times\ldots\times[0,\lceil 2\pi L_d\rceil -1])$ so that
\begin{equation}\label{eq:inclusions}
\bigcup_{j\in\hat\Gamma}\left([0,1]^d+j\right)\subset \TT_L^d \subset \bigcup_{j\in\tilde\Gamma}\left([0,1]^d+j\right). \\
\end{equation}
In general the set on the right is larger than $\TT_L^d$ (consider e.\,g.~$d=1$ and $2 \pi L_1 = 1+\epsilon$). For this reason the factor $2^d$ appears in 
\begin{equation*}
\sum_{j\in \tilde\Gamma}\norm{f}{L^p([0,1]^d+j)}^p \leq 2^d\norm{f}{L^p(\TT_L^d)}^p.
\end{equation*}

Note that if $\abs{f}$ is constant one use the first inclusion in 
\eqref{eq:inclusions} to conclude for all  $p \in [1,\infty)$
\[
 \int_{\TT_L^d} \abs{f}^p \leq \frac{2^d}{\gamma}\int_{S\cap\TT_L^d} \abs{f}^p. 
\]

In what follows we will exclude the case of constant $\abs{f}$  and will denote any of the cubes $[0,1]^d+j$ by $\Lambda$.

We now claim that for any point $y$ in a cube $\Lambda$ there exists a line segment $I:=I(\Lambda, S,y)\subset \Lambda$ such that $y\in I$
and $\frac{\abs{S\cap I}}{\abs{I}}\geq \frac{\gamma}{C_1^d}$. In fact, using spherical coordinates around $y$ we want to sweep the set $S \cap \Lambda$ by line segments starting at $y$.
To implement this we write
\begin{equation}\label{eq:eq:13}
  \abs{S\cap \Lambda}=\int_{\Lambda\cap S} \dd x=\int_{\abs{\xi}=1}\int_{0}^{\infty}\chi_{S\cap \Lambda}(y+r\xi)r^{d-1}\dd r \dd\sigma(\xi).
\end{equation}

We set $\sigma_{d-1}=\abs{\mathbb{S}^{d-1}}$. 
Then, there exists a point $\eta= \eta(S)\in\mathbb{S}^{d-1}$ such that
\begin{equation}\label{eq:*}
\abs{S\cap \Lambda}\leq \sigma_{d-1}\int_0^{\infty}\chi_{S\cap \Lambda}(y+r\eta)r^{d-1}\dd r,
\end{equation}
since the maximum cannot be smaller than the average.

We define the line segment $I$ to be the longest line segment in $\Lambda$ starting at $y$ in the direction $\eta=\eta(S)$, that is,
\begin{equation}\label{eq:line-segment-2}
I=I(\Lambda, S,y, \eta)= \{x\in \Lambda\quad\vert\quad x=y+r\eta,\quad r\geq 0\}.
\end{equation}

Consequently, Ineq. \eqref{eq:*} yields
\begin{equation}\label{eq:eq:14}
\abs{S\cap \Lambda}
\leq \sigma_{d-1} d^{(d-1)/2}\int_0^{\infty}\chi_{S\cap I}(y+r\eta)\dd r
=\sigma_{d-1} d^{(d-1)/2}\abs{S\cap I}.
\end{equation}

Since 
\[
\sigma_{d-1}=2^{-1}\pi^{\frac{d}{2}-1}e^{\frac{d}{2}+\frac{1}{2}}\left(\frac{d+1}{2}\right)^{-d/2}\left(1+O(\frac{2}{d+1})\right) 
\]
using the asymptotic approximation of the Gamma function, 
there exists a numerical constant $C_1>1$ such that $\sigma_{d-1}d^{d/2}\leq C_1^d$. 
This fact, together with the above inequality \eqref{eq:eq:14} and $\abs{I}\leq d^{1/2}$, yields
\begin{equation}\label{eq:line-segment}
\frac{\abs{S\cap I}}{\abs{I}}
\geq \frac{\abs{S\cap \Lambda}}{\sigma_{d-1}d^{d/2}}
\geq \frac{\abs{S\cap \Lambda}}{C_1^d}
\geq\frac{\gamma}{C_1^d}.
\end{equation}

Let $y_0\in \Lambda$
be a point such that $\abs{f(y_0)}\geq\norm{f}{L^p(\Lambda)}$, e.g.~a maximum of $f$ in $\Lambda$, and define $F_\eta\colon \CC\longrightarrow\CC$
by $F_\eta(w):=\norm{f}{L^p(\Lambda)}^{-1}f(y_0+w\abs{I_0}\eta)$, where $I_0:=I(\Lambda, S,y_0, \eta)$ is as in \eqref{eq:line-segment-2}.

Setting $D(\xi, R)=\{\nu\in\CC\;\vert\;\abs{\nu-\xi}\leq R\}$ for $\xi\in\CC$ and 
$D:= D(0, 4)\times\ldots\times D(0, 4)$, we observe 
\begin{equation}\label{eq:max_F}
 M_\eta:=\max_{\abs{w}\leq 4} \abs{F_\eta(w)}
 \leq \max_{\xi\in \mathbb{S}^{d-1}}\max_{\abs{w}\leq 4} \abs{F_\xi(w)}
 \leq 
 \max_{z \in D}  \frac{\abs{f(y_0+z)}}{\norm{f}{L^p(\Lambda)}}=:M, 
\end{equation}
since $\abs{\abs{I}\eta_i}\leq 1$ for all $i=1,\ldots, d$, for all $\eta\in\mathbb{S}^{d-1}$ 
and all $I\subset \Lambda$ line segments. 
Note that $\log M >0$ since we assumed that $\abs{f}$ is not constant. 

By definition of $F_\eta$ and for $A:=\{t\in[0,1]\;\vert\;y_0+t\abs{I_0}\eta\in S\cap I_0\}$ we have
\begin{equation}\label{eq:eq:15a}
\sup_{x\in S\cap \Lambda}\abs{f(x)}\geq\sup_{x\in S\cap I_0}\abs{f(x)}=\norm{f}{L^p(\Lambda)}\sup_{t\in A}\abs{F_\eta(t)}.
\end{equation}
We now apply Theorem \ref{lemma:2} to the function $F_\eta$, the interval $[0,1],$ and the subset $A$, 
whose measure is $\abs{A}=\frac{\abs{S\cap I_0}}{\abs{I_0}}$:
\begin{equation}\label{eq:eq:15b}
\norm{f}{L^p(\Lambda)}\sup_{t\in A}\abs{F_\eta(t)}
\geq\norm{f}{L^p(\Lambda)}\Big(\frac{\abs{A}}{12}\Big)^{\frac{2\log M_\eta}{\log 2}}\sup_{t\in [0,1]}\abs{F_\eta(t)}.
\end{equation}
The last term is bounded below by
\begin{equation}\label{eq:eq:15c}
\Big(\frac{\abs{A}}{12}\Big)^{\frac{2\log M}{\log 2}}\abs{f(y_0)}\\
\geq\Big(\frac{\abs{S\cap I_0}}{12\abs{I_0}}\Big)^{\frac{2\log M}{\log 2}}\norm{f}{L^p(\Lambda)}\geq\Big(\frac{\abs{S\cap \Lambda}}{C_2^d}\Big)^{\frac{2\log M}{\log 2}}\norm{f}{L^p(\Lambda)},
\end{equation}
where we used the choice of $y_0$, inequalities \eqref{eq:line-segment} and \eqref{eq:max_F}, 
and where $C_2$ is a numerical constant such that $C_2^d \geq 12 C_1^d$. 
Similarly, for
\[
{W}=\{x\in \Lambda\;\vert\;\abs{f(x)}<\left(\frac{\abs{S \cap \Lambda}}{1+C_2^d}\right)^{2\log M/\log 2}\norm{f}{L^p(\Lambda)}\}
\]
and the point $y_0$ chosen above, we proceed as in \eqref{eq:eq:13}--\eqref{eq:line-segment} 
substituting $S$ with $W$ to find a (new) direction $\tilde{\eta}=\eta(W)$  
and a corresponding line segment 
$\tilde I_0:=I(\Lambda, W, y_0, \tilde{\eta})$ containing $y_0$ and  
satisfying $\frac{\abs{W\cap \tilde I_0}}{\abs{\tilde I_0}}\geq \frac{\abs{W}}{C_1^d}$. 
Defining the modified function $\tilde F_{\tilde{\eta}}:\CC \rightarrow \CC$, 
$\tilde F_{\tilde{\eta}}(w)=\norm{f}{L^p(\Lambda)}^{-1}f(y_0+w\abs{\tilde I_0}\tilde\eta)$ and applying Theorem \ref{lemma:2} 
to the function $\tilde F_{\tilde{\eta}}$ with $M_{\tilde{\eta}}=\max_{\abs{w}\leq 4}\abs{\tilde F_{\tilde\eta}(w)}\leq M$, 
the set $[0,1]$, and $\tilde A:=\{t\in[0,1]\;\vert\;y_0+t\abs{\tilde I_0}\tilde\eta\in W\cap \tilde I_0\}$ 
satisfying $\abs{\tilde A}=\frac{\abs{W\cap \tilde I_0}}{\abs{\tilde I_0}}$, we follow the same steps of \eqref{eq:eq:15a}--\eqref{eq:eq:15c} to conclude
\begin{equation}\label{eq:eq:15'}
\begin{split}
\sup_{x\in W}\abs{f(x)}
&\geq\norm{f}{L^p(\Lambda)}\Big(\frac{\abs{\tilde A}}{12}\Big)^{\frac{2\log M_{\tilde\eta}}{\log 2}}\sup_{t\in [0,1]}\abs{\tilde F_{\tilde\eta}(t)}
\geq\Big(\frac{\abs{\tilde A}}{12}\Big)^{\frac{2\log M}{\log 2}}\abs{f(y_0)}\\
&\geq\Big(\frac{\abs{W\cap \tilde I_0}}{12\abs{\tilde I_0}}\Big)^{\frac{2\log M}{\log 2}}\norm{f}{L^p(\Lambda)}\geq\Big(\frac{\abs{W}}{C_2^d}\Big)^{\frac{2\log M}{\log 2}}\norm{f}{L^p(\Lambda)},
\end{split}
\end{equation}

Lemma \ref{lemma:level-set-argument} applied with $Q=0$, $U=\Lambda\cap S$ and $\alpha=2\log M/\log 2>0$ gives
\begin{equation}\label{eq:1}
\begin{split}
\norm{f}{L^p(S\cap \Lambda)}&\geq\left(\frac{\abs{\Lambda\cap S}}{1+C_2^d}
\right)^{\frac{2\log M}{\log 2}+\frac{1}{p}}\norm{f}{L^p(\Lambda)}
\geq\left(\frac{\gamma}{C_3^d}\right)^{\frac{2\log M}{\log 2}+\frac{1}{p}}\norm{f}{L^p(\Lambda)},
\end{split}
\end{equation}
for some numerical constant $C_3$ so that $1+C_2^d\leq C_3^d$.

One is now left with estimating $M=
\max_{z \in D}  \frac{\abs{f(y_0+z)}}{\norm{f}{L^p(\Lambda)}}$,
which depends on the particular cube $\Lambda=[0,1]^d+j$ under consideration.
We do not know how to estimate this quantity for \emph{every} cube $\Lambda$, but for sufficiently many of them.
To make this precise, we consider two types of cubes.

Motivated by the Bernstein Inequality (Proposition \ref{prop:bernstein}), 
we follow Kovrijkine \cite[p.~6]{Kovrijkine-thesis} and call a cube $\Lambda$ \emph{good} if
for all multi-indices $\alpha\in\NN_0^d\setminus\{0\}$
\begin{equation}\label{eq:bad}
\norm{\partial^{\alpha}f}{L^p(\Lambda)}< 2^{(2d)/p}(3C_B b)^{\alpha}\norm{f}{L^p(\Lambda)},
\end{equation}
where $C_B$ is the constant from \eqref{eq:bernstein}. Otherwise, we call $\Lambda$ \emph{bad}.

It is now easy to check that
\begin{equation}\label{eq:bad-mass}
\norm{f}{L^p(\bigcup\limits_{\Lambda\text{ bad}}\Lambda)}^p\leq\frac{1}{2}\norm{f}{L^p(\TT_L^d)}^p,
\end{equation}
and therefore good cubes exist and the contribution of the bad cubes can be subsumed in the contribution of the good ones.

Indeed, using the definition of bad cubes and the Bernstein Inequality, we have
\begin{multline}\label{eq:eq:10}
\norm{f}{L^p(\bigcup\limits_{\Lambda\text{ bad}}\Lambda)}^p
\leq
\sum_{\alpha\neq 0}\sum_{\Lambda\text{ bad}}\frac{1}{2^{2d}(3C_Bb)^{p\alpha}}\norm{\partial^\alpha f}{L^p(\Lambda)}^p
\\
\leq\sum_{\alpha\neq 0}\frac{2^d}{2^{2d}(3C_Bb)^{p\alpha}}\norm{\partial^\alpha f}{L^p(\TT_L^d)}^p
\leq
\sum_{\alpha\neq 0}\frac{1}{2^{d}3^{p\abs{\alpha}}}\norm{f}{L^p(\TT_L^d)}^p\\
=
\frac{1}{2^d}\left(\frac{1}{\left(1-\frac{1}{3^p}\right)^d}-1\right)\norm{f}{L^p(\TT_L^d)}^p
< 
\frac{1}{2}\norm{f}{L^p(\TT_L^d)}^p,
\end{multline}
which gives Ineq. \eqref{eq:bad-mass}.

We now claim that in any good cube $\Lambda$ there exists a point $x\in \Lambda$ such that
\begin{equation}\label{eq:claim}
\abs{\partial^\alpha f(x)}\leq {2^{(3d)/p}}(9C_Bb)^{\alpha}\norm{f}{L^p(\Lambda)}\qquad\forall\;\alpha\in\NN_0^d.
\end{equation}
The proof follows from a contradiction argument. We assume that for every $x\in\Lambda$, with $\Lambda$ a good cube,
there exists $\alpha(x)\in\NN_0^d$ such that
\begin{equation}
\abs{\partial^{\alpha(x)} f(x)}^p> 2^{3d} (9C_Bb)^{p\alpha(x)} \norm{f}{L^p(\Lambda)}^p
\end{equation}
To get rid of the $x$-dependence in $\alpha(x)$ we sum over all multi-indices to obtain          
\begin{equation*}
\sum_{\alpha\in\NN_0^d}  \frac{\abs{\partial^\alpha f(x)}^p}{ 2^{(3d)} (9C_Bb)^{p\alpha} }
\geq
\frac{\abs{\partial^{\alpha(x)} f(x)}^p}{ 2^{(3d)} (9C_Bb)^{p\alpha(x)} }
> \norm{f}{L^p(\Lambda)}^p.
\end{equation*}
Integration over $\Lambda$ and the definition of good cubes yield
\begin{equation}\label{eq:eq:11}
\begin{split}
2^{d}\norm{f}{L^p(\Lambda)}^p&\leq\sum_{\alpha\in\NN^d_0}\frac{1}{2^{2d}(9C_Bb)^{p\alpha}}\norm{\partial^\alpha f}{L^p(\Lambda)}^p\\
&\leq\sum_{\alpha\in\NN^d_0}\Big(\frac{1}{3}\Big)^{p\abs{\alpha}}\norm{f}{L^p(\Lambda)}^p
=\frac{1}{\Big(1-(1/3)^p\Big)^d}\norm{f}{L^p(\Lambda)}^p.
\end{split}
\end{equation}
Then,
\begin{equation}\label{eq:eq:12}
\norm{f}{L^p(\Lambda)}^p\leq\frac{1}{2^{d}\Big(1-(1/3)^p\Big)^d}\norm{f}{L^p(\Lambda)}^p < \norm{f}{L^p(\Lambda)}^p,
\end{equation}
yielding a contradiction. Therefore, the claim holds true.

Let now $\Lambda=[0,1]^d+j$ for some $j\in\tilde \Gamma$ be a good cube. 
Setting $r_i=j_i+1/2$ for all $i=1,\ldots,d$, we have 
\begin{equation} \label{eq:lambda-cube}
\Lambda=[0,1]^d+j = [r_1-1/2,r_1+1/2]\times\ldots\times[r_d-1/2,r_d+1/2].
\end{equation}
Since  $y_0\in \Lambda$, if $z\in D$ (see the text before inequality \eqref{eq:max_F} for the definition of the polydisc $D$), 
then $\tilde z=y_0+ z \in D(r_1,4+1/2)\times\ldots\times D(r_d,4+1/2)=:\widetilde{D}$.
For  $x\in\Lambda$ as in \eqref{eq:claim}, we have that
$\widetilde{D}\subset D(x_1,5)\times\ldots\times D(x_d,5)$
and for any $\tilde z\in\widetilde{D}$, the series expansion of $f$ gives
\begin{equation}\label{eq:taylor}
\begin{split}
\abs{f(\tilde z)}\leq\sum_{\alpha\in\NN^d_0}\frac{\abs{\partial^\alpha f(x)}}{\alpha!}\abs{\tilde z-x}^\alpha
&\leq\sum_{\alpha\in\NN^d_0}{2^{(3d)/p}}(\widetilde{C}b)^{\alpha}5^{\abs{\alpha}}\frac{1}{\alpha!}\norm{f}{L^p(\Lambda)}\\
&={2^{(3d)/p}}\exp\left(5\widetilde{C} \abs{b}_1\right)\norm{f}{L^p(\Lambda)},
\end{split}
\end{equation}
where $\abs{b}_1= \sum_{j=1}^d b_j$, $\tilde{C}=9C_B$, and where we used \eqref{eq:claim} in the second inequality .

We are now in position to bound $M$ defined in \eqref{eq:max_F} associated with the good cube $\Lambda$. By \eqref{eq:taylor}
\begin{equation}\label{eq:M-bound}
M=
 \max_{z \in D}  \frac{\abs{f(y_0+z)}}{\norm{f}{L^p(\Lambda)}}
\leq\norm{f}{L^p(\Lambda)}^{-1}\max_{\tilde z\in\widetilde{D} }\abs{f(\tilde z)}\leq {2^{(3d)/p}}\exp\left(5\widetilde{C}\abs{b}_1\right),
\end{equation}
and consequently,
\begin{equation}\label{eq:M-bound-2}
\log M\leq\left(\frac{3d}{p}\right)\log 2+5\widetilde{C}\abs{b}_1.
\end{equation}

Substituting \eqref{eq:M-bound-2} into \eqref{eq:1} and summing over all good cubes $\Lambda$, we have
\begin{equation}\label{eq:semi-final}
\norm{f}{L^p(S\cap\TT_L^d)}
\geq\frac{1}{2}\left(\frac{\gamma}{C_3^d}\right)^{\frac{6d+1}{p}+C_4\abs{b}_1}\norm{f}{L^p(\TT_L^d)}, \quad C_4=90 \, C_B/\log 2,
\end{equation}
and this concludes the proof for $p\in[1,\infty), a=(1,\ldots,1)$ and $2\pi L_j\geq 1$ for all $j=1,\ldots, d$.

If $p=\infty$, the proof follows the same steps. However, in this case a cube $\Lambda$ is
\textsl{good} if           
\begin{equation}\label{eq:bad-infty}
\norm{\partial^{\alpha}f}{L^\infty(\Lambda)}< 2^{2d}(3C_B b)^{\alpha}\norm{f}{L^\infty(\Lambda)}\quad \forall\,\alpha\in\NN_0^d\setminus\{0\}
\end{equation}
and we have that $\norm{f}{L^\infty(\bigcup\limits_{\Lambda\text{ good}}\Lambda)}=\norm{f}{L^\infty(\TT_L^d)}$. Then, we conclude
\begin{equation}\label{eq:eq:17}
\norm{f}{L^\infty(S\cap\TT_L^d)}\geq\left(\frac{\gamma}{\tilde{C}_3^d}\right)^{6d+\tilde{C_4}\abs{b}_1}\norm{f}{L^\infty(\TT_L^d)},
\end{equation}
where $\tilde C_3$ and $\tilde C_4$ are numerical constants. 

Finally, we treat the general case by rescaling.

Let us assume that $L\in\RR_+^d$, the vector $a=(a_1,\ldots,a_d)$ has components $a_j\leq 2\pi L_j$ for all $j=1,\ldots,d$,
$S$ is a $(\gamma, a)$-thick set, and $f\in L^p(\TT_L^d)$ with $\supp\h{f}\subset[-b_1/2+\lambda_1,b_1/2+\lambda_1]\times\ldots\times[-b_d/2+\lambda_d,b_d/2+\lambda_d]$.

We define the transformation $T(x_1,\ldots,x_d)=(a_1x_1,\ldots,a_dx_d)$ for all $x\in\RR^d$.
In particular, $T(\TT^d_{L/a})=\TT_L^d$ and $F=T^{-1}(S)$ is $(\gamma, 1)$-thick.
Further, we consider the function $g=f\circ T:\TT^d_{L/a}\longrightarrow\CC$.
Then, it holds $\h{f}(\frac{k_1}{L_1},\ldots,\frac{k_d}{L_d})=\h{g}(\frac{a_1k_1}{L_1},\ldots, \frac{a_dk_d}{L_d})$, which implies
\begin{equation}\label{eq:support-scaling}
\supp\h{g}\subset \left[a_1\left(\lambda_1-\frac{b_1}{2}\right),a_1\left(\lambda_1+\frac{b_1}{2}\right)\right]\times\ldots\times\left[a_d\left(\lambda_d-\frac{b_d}{2}\right), a_d\left(\lambda_d+\frac{b_d}{2}\right)\right].
\end{equation}
Moreover, for $p<\infty$ we have
\begin{equation}\label{eq:p-norms}
\left(\prod_{j=1}^d a_j\right)\norm{g}{L^p(\TT^d_{L/a})}^p=\norm{f}{L^p(\TT_L^d)}^p,
\quad
\left(\prod_{j=1}^d a_j\right)\norm{g}{L^p(\TT^d_{L/a}\cap F)}^p=\norm{f}{L^p(\TT_L^d\cap S)}^p.
\end{equation}
Even simpler are the relations for $p=\infty$:
\begin{equation}
\norm{g}{L^\infty(\TT^d_{L/a})}=\norm{f}{L^\infty(\TT_L^d)},
\qquad
\norm{g}{L^\infty(\TT^d_{L/a}\cap F)}=\norm{f}{L^\infty(\TT_L^d\cap S)}.
\end{equation}

Therefore, applying \eqref{eq:semi-final} to $g$, $F$ and $\TT^d_{L/a}$ and using \eqref{eq:support-scaling}
and the relation between the $L^p$ norms of $f$ and $g$ we conclude
\begin{multline} 	\label{eq:eq:18}
\norm{f}{L^p(S\cap\TT_L^d)}^p=\left(\prod_{j=1}^d a_j\right)\norm{g}{L^p(\TT^d_{L/a}\cap F)}^p
\\
\geq\left(\prod_{j=1}^d a_j\right)\left(\frac{\gamma}{C_3^d}\right)^{6d+1+pC_4 a\cdot b}\norm{g}{L^p(\TT^d_{L/a})}^p
=\left(\frac{\gamma}{C_3^d}\right)^{6d+1+pC_4 a\cdot b}\norm{f}{L^p(\TT_L^d)}^p,
\end{multline}
and similarly for the $L^\infty$-norm case.

\begin{remark}
If each $2\pi L_j$ is an integer multiple of $a_j$, then the exponent in Theorem \ref{thm:1} becomes $c a\cdot b  + (4p+1)/p$.
This is due to the fact that the union of cubes with unit sidelengths covers the torus without overlapping.
Therefore, one can reduce the factors in \eqref{eq:bad} and \eqref{eq:claim} to $2^{d/p}$ and $2^{2d/p}$ respectively,
which leads to the stated exponent.
\end{remark}

\section{Proof of Theorem \ref{thm:2}} \label{sec:proofKLS}

The proof of Theorem \ref{thm:2} is divided into two cases. First we assume the
 $d$-dimensional rectangles $J_l$ to be disjoint,
then we recover the general case by an induction argument on the number of  $d$-dimensional rectangles.
The general strategy of the proof
resembles ideas from the proof of Theorem \ref{thm:1}.
In fact, we split $\TT_L^d$ into unit cubes as done before and we aim at obtaining a local estimate on each cube. Then, summing over all cubes
yields the result. We point out that in this case we do not classify the cubes into two types
as we need other estimates due to the nature of the function $f$, see \eqref{eq:function}.

Let the  $d$-dimensional rectangles $J_l$ be all disjoint.

In addition, we first assume that $p\in[1,\infty)$, $a=(1,\ldots,1)$, and $2\pi L_j\geq 1$ for all $j=1,\ldots,d$.

By Lemma \ref{lemma:sum-estimate} we write
\begin{equation}\label{eq:function}
f(x)=\sum_{l=1}^n f_l(x)e^{i c_l\cdot x},
\end{equation}
for some functions $f_l$ with $\supp\h{f_l}\subset [-b_1,b_1]\times\ldots\times[-b_d,b_d]$
and some points $c_l\in \left(\frac{1}{L_1}\ZZ\times\ldots\times\frac{1}{L_d}\ZZ\right)\cap J_l$.
We also have
\begin{equation}\label{eq:claim-2}
\norm{f_l}{L^p(\TT_L^d)}\leq K^d\norm{f}{L^p(\TT_L^d)}\qquad\forall\, l=1,\dots, n.
\end{equation}

As in the proof of Theorem \ref{thm:1}, we now cover the torus by unit cubes
\begin{equation*}
\TT_L^d\subset \bigcup_{j\in \tilde\Gamma}\left([0,1]^d+j\right),\quad \tilde\Gamma=\ZZ^d\cap ([0,\lceil 2\pi L_1\rceil -1]\times\ldots\times[0,\lceil 2\pi L_d\rceil -1]),
\end{equation*}
which yields $\sum_{j\in \tilde\Gamma}\norm{f}{L^p([0,1]^d+j)}^p\leq 2^d\norm{f}{L^p(\TT_L^d)}^p$.
Again, we denote by $\Lambda$ any cube of the form $[0,1]^d+j$.

Let $y_0\in \Lambda$ be such that $f(y_0)=\max_{x\in\Lambda}\abs{f(x)}$.
By the same probabilistic argument in the proof of Theorem \ref{thm:1},
there exists a line segment $I_0=I(\Lambda, S,y_0,\eta)=\{x\in \Lambda\;\vert\; x=y_0+r\eta, r\geq 0 \}\subset \Lambda$
such that $y_0\in I_0$ and
$\frac{\abs{S\cap I_0}}{\abs{I_0}}\geq\frac{\abs{\Lambda\cap S}}{C_1^d}\geq\frac{\gamma}{C_1^d}$.

We now define $G_\eta:\RR\longrightarrow\CC$ as
\begin{equation}\label{eq:eq:31}
  G_\eta(t)=f(y_0+t\abs{I_0}\eta)=\sum_{l=1}^n f_l(y_0+t\abs{I_0}\eta)e^{i c_l\cdot(y_0+t\abs{I_0}\eta)},
\end{equation}
and applying Taylor expansion with remainder term in integral form we obtain
\begin{align}\nonumber 
G_\eta(t)&=\sum_{l=1}^n p_l(t)e^{i c_l\cdot(y_0+t\abs{I_0}\eta)}
+\sum_{l=1}^n\frac{e^{i c_l\cdot(y_0+t\abs{I_0}\eta)}}{(m-1)!}\int_0^t(t-s)^{m-1}\frac{d^{m}}{ds^m}f_l(y_0+s\abs{I_0}\eta)\dd s\\
&=P_\eta(t)+r_\eta(t),
\label{eq:eq:19}
\end{align}
where the $p_l(t)$'s are polynomials of degree at most $m-1$.

Then, Theorem \ref{lemma:3} applied to $P_\eta$, $[0,1]$, and $A=\{t\in[0,1]\;\vert\;y_0+t\abs{I_0}\eta\in S\cap I\}$, whose measure is
$\abs{A}=\frac{\abs{S\cap I_0}}{\abs{I_0}}\geq\frac{\abs{S\cap\Lambda}}{C_1^d}$, yields
\begin{equation}\label{eq:eq:22a}
\begin{split}
\norm{f}{L^\infty(\Lambda)}& 
=\norm{G_\eta}{L^{\infty}([0,1])}
\leq\norm{P_\eta}{L^\infty([0,1])}+\norm{r_\eta}{L^\infty([0,1])}
\\ &
\leq\Big(\frac{\tilde{C_3}}{\abs{A}}\Big)^{nm-1}\norm{P_\eta}{L^\infty(A)}+\norm{r_\eta}{L^\infty([0,1])}\\
\end{split}
\end{equation}
Using the estimate on $| A |$, we see that the right hand side of \eqref{eq:eq:22a} is bounded above by
\begin{multline}\label{eq:eq:22b}
\Big(\frac{\tilde{C_3}C_1^d}{\abs{S\cap\Lambda}}\Big)^{nm-1}\norm{P_\eta}{L^\infty(A)}
+\norm{r_\eta}{L^\infty([0,1])}\\
\leq\Big(\frac{\tilde{C_3}C_1^d}{\abs{S\cap\Lambda}}\Big)^{nm-1}\norm{G_\eta}{L^\infty(A)}+\left[\left(\frac{\tilde{C_3}C_1^d}{\abs{S\cap\Lambda}}\right)^{nm-1}+1\right]\norm{r_\eta}{L^\infty([0,1])}\\
\leq\left(\frac{\tilde{C_4}^d}{\abs{S\cap\Lambda}}\right)^{nm-1}\norm{f}{L^\infty(\Lambda\cap S)}+\left(\frac{\tilde{C_4}^d}{\abs{S\cap\Lambda}}\right)^{nm-1}\norm{r_\eta}{L^\infty([0,1])}.
\end{multline}

We now estimate the $L^\infty$-norm of the remainder term $r_\eta$ and show that the bound depends on the cube $\Lambda$ but not
on $\eta$. 
\begin{multline}\label{eq:eq:20a}
\max_{t\in[0,1]}\abs{r_\eta(t)}
\leq\max_{t\in[0,1]}\sum_{l=1}^n\frac{1}{(m-1)!}
\int_0^t\left\vert(t-s)^{m-1}\frac{d^{m}}{ds^m}f_l(y_0+s\abs{I_0}\eta)\right\vert\dd s\\
\leq\max_{t\in[0,1]}\sum_{l=1}^n\frac{1}{(m-1)!}
\max_{\tau\in[0,1]}\left\vert\frac{d^{m}}{d\tau^m}f_l(y_0+\tau\abs{I_0}\eta)\right\vert
\cdot  \int_0^t (t-s)^{m-1}\dd s  \\
\leq\frac{1}{m!}\sum_{l=1}^n\max_{\tau\in[0,1]}\Big\vert\frac{m!}{\alpha!}\mathop{\sum_{\alpha\in\NN^d_0,}}_{\abs{\alpha}=m}\partial^\alpha f_l(y_0+\tau\abs{I_0}\eta)(\abs{I_0}\eta)^{\alpha}\Big\vert\\
\end{multline}
Since $\abs{\abs{I_0}\eta_i}\leq 1$ for all $i=1,\ldots, d$ (see also the estimates after \eqref{eq:max_F} in Section \ref{s:proofLS}) we obtain
$\abs{(\abs{I_0}\eta)^\alpha}\leq 1$. Hence the right hand term in \eqref{eq:eq:20a} is bounded above by
\begin{equation}\label{eq:eq:20b}
\frac{1}{m!}\sum_{l=1}^n\mathop{\sum_{\alpha\in\NN^d_0,}}_{\abs{\alpha}=m}\frac{m!}{\alpha!}\max_{x\in \Lambda}\abs{\partial^\alpha f_l(x)}
\leq \frac{1}{m!}\sum_{l=1}^n\mathop{\sum_{\alpha\in\NN^d_0,}}_{\abs{\alpha}=m}\frac{m!}{\alpha !}\mathop{\sum_{\beta\in\NN_0^d,}}_{ \beta_j\in\{0,1\}}\norm{\partial^{(\alpha+\beta)}f_l}{L^1(\Lambda)}=: M_\Lambda,
\end{equation}
where we applied  the estimate
$\abs{\psi(x)}\leq\sum\limits_{\beta\in\NN_0^d,\beta_j\in\{0,1\}}\int_\Lambda\abs{\partial^\beta \psi}$ to the function $\psi=\partial^\alpha f_l$.

From Ineq. \eqref{eq:eq:22a}--\eqref{eq:eq:20b}, it follows
\begin{equation}\label{eq:eq:double}
\norm{f}{L^\infty(\Lambda)}\leq\left(\frac{\tilde{C_4}^d}{\abs{S\cap\Lambda}}\right)^{nm-1}\left(\norm{f}{L^\infty(\Lambda\cap S)}+M_\Lambda\right).
\end{equation}

Similarly, for the set 
\[
W:=\{x\in\Lambda\;\vert\;\abs{f(x)} + M_\Lambda <\left(\frac{\abs{S\cap\Lambda}}{\tilde{C_4}^d+1}\right)^{nm-1}\norm{f}{L^p(\Lambda)}\},
\]
there exists a (possibly different) direction $\tilde{\eta}=\eta(W)$ and (possibly different) 
line segment $I=I(\Lambda,W, y_0, \tilde\eta)$ satisfying $y_0\in I$ and
$\frac{\abs{W\cap I}}{\abs{I}}\geq\frac{\abs{\Lambda\cap W}}{C_1^d}=\frac{\abs{W}}{C_1^d}$.
Then, by using the   function $\tilde G_{\tilde{\eta}}\colon \RR \rightarrow \CC$,
$\tilde G_{\tilde\eta}(t)=f(y_0+t\abs{I}\tilde\eta)$,
analogous to \eqref{eq:eq:31}, 
and repeating the steps \eqref{eq:eq:19}--\eqref{eq:eq:20b}, 
 inequality \eqref{eq:eq:double} holds with $W$ instead of $S\cap \Lambda$.
We remark once more that the estimate of the $L^\infty$-norm of $r_{\tilde\eta}$ remains the same:
The function $\tilde G_{\tilde\eta}$, and consequently $\tilde P_{\tilde\eta}$ and $\tilde r_{\tilde\eta}$, depend on the choice of $\tilde{\eta}\in\mathbb{S}^{d-1}$, but the $L^\infty$ estimate does not,
as we derive it using properties that any line segment contained in $\Lambda$ has.
Now, Lemma \ref{lemma:level-set-argument} applied with $Q=M_\Lambda$, $\alpha=nm-1$, $q=p$,
and $U= S\cap\Lambda$ yields
\begin{equation}\label{eq:eq:23}
\begin{split}
\norm{f}{L^p(\Lambda)}^p&\leq \left(\frac{1+\tilde{C_4}^d}{\abs{\Lambda\cap S}}\right)^{pnm-(p-1)}\norm{\abs{f}+M_\Lambda}{L^p(\Lambda\cap S)}^p\\
&\leq\left(\frac{C_5^d}{\gamma}\right)^{pnm-(p-1)}\left(2^{p-1}\norm{f}{L^p(\Lambda\cap S)}^p+2^{p-1}\abs{\Lambda\cap S }M_\Lambda^p\right)\\
&\leq 2^{p-1}\left(\frac{C_5^d}{\gamma}\right)^{pnm-(p-1)}\norm{f}{L^p(\Lambda\cap S)}^p+ 2^{p-1}\left(\frac{C_5^d}{\gamma}\right)^{pnm}M_\Lambda^p,\\
\end{split}
\end{equation}
where we use $\abs{\Lambda\cap S}\leq 1$ in the last line.

Summing over all cubes $\Lambda$ we obtain
\begin{equation}\label{eq:eq:sum2}
\int_{\TT_L^d}\abs{f}^p\leq
2^{d}2^{p-1}\Big(\frac{C_5^d}{\gamma}\Big)^{pnm-(p-1)}\int_{\TT_L^d\cap S}\abs{f}^p
+2^{p-1}\Big(\frac{C_5^d}{\gamma}\Big)^{pnm}\sum_{\Lambda\text{ cubes}} M_\Lambda^p,
\end{equation}
and to conclude the proof we are left with estimating the sum on the right hand side.

We first estimate $M_\Lambda^p$ applying H\"older's Inequality to the $L^1$-norm of $\partial^{(\alpha+\beta)} f$
and to the three sums on the right side of \eqref{eq:eq:20b}. We obtain
\begin{equation}\label{eq:eq:21}
\begin{split}
M_\Lambda^p & \leq\frac{(2^d (m+1)^{d-1} n)^{p-1}}{(m!)^p}\sum_{l=1}^n\mathop{\sum_{\alpha\in\NN^d_0,}}_{\abs{\alpha}=m}\mathop{\sum_{\beta\in\NN_0^d,}}_{ \beta_j\{0,1\}}\Big(\frac{m!}{\alpha!}\Big)^p\norm{\partial^{(\alpha+\beta)}f_l}{L^p(\Lambda)}^p\\
& \leq \frac{(4^d m^d n)^{p-1}}{(m!)^p}\sum_{l=1}^n\mathop{\sum_{\alpha\in\NN^d_0,}}_{\abs{\alpha}=m}\mathop{\sum_{\beta\in\NN_0^d,}}_{\beta_j\{0,1\}}\Big(\frac{m!}{\alpha!}\Big)^p\norm{\partial^{(\alpha+\beta)}f_l}{L^p(\Lambda)}^p.
\end{split}
\end{equation}

We now take the sum over all cubes $\Lambda$ and use the Bernstein Inequality \eqref{eq:bernstein} to get	
\begin{equation}\label{eq:rest}
\begin{split}
&\sum_{\Lambda\text{ cubes}} M_\Lambda^p
\leq\frac{(4^d m^d n)^{p-1}}{(m!)^p}\sum_{l=1}^n\mathop{\sum_{\alpha\in\NN^d_0,}}_{\abs{\alpha}=m}\mathop{\sum_{\beta\in\NN_0^d,}}_{\beta_j\{0,1\}}\Big(\frac{m!}{\alpha!}\Big)^p2^{d}\int_{\TT_L^d}\abs{\partial^{(\alpha+\beta)}f_l}^p\\
&\leq\frac{(4^d m^d n)^{p-1}2^{d}}{(m!)^p}\sum_{l=1}^n\mathop{\sum_{\alpha\in\NN^d_0,}}_{\abs{\alpha}=m}\mathop{\sum_{\beta\in\NN_0^d,}}_{\beta_j\in\{0,1\}}\Big(\frac{m!}{\alpha!}\Big)^p(C_B b)^{(\alpha+\beta)p}\int_{\TT_L^d}\abs{f_l}^p\\
&\leq\frac{(4^{d} m^{d} n)^{p-1}2^{d}}{(m!)^p}\Bigg(\mathop{\sum_{\beta\in\NN^d_0,}}_{\beta_j\in\{0,1\}}\mathop{\sum_{\alpha\in\NN^d_0,}}_{\abs{\alpha}=m}\frac{(m!)(C_B b)^{(\alpha+\beta)}}{\alpha! \beta!}\Bigg)^p \sum_{l=1}^n\int_{\TT_L^d}\abs{f_l}^p.
\end{split}
\end{equation}

Using \eqref{eq:claim-2} on each $f_l$, the equality
\[
\sum_{\alpha\in\NN^d_0,\abs{\alpha}=m}\frac{(m!)(C_B b)^{\alpha}}{\alpha! } = (C_B \abs{b}_1)^m,
\]
which follows by an induction argument, and noting that the $\beta$-sum is the truncated power series
expansion of $e^{C_B\abs{b}_1}$, we conclude
\begin{equation}\label{eq:32}
\begin{split}
\sum_{\Lambda\text{ cubes}}& M_\Lambda^p\leq \frac{(4^d m^{d}n)^{p-1}2^{d}nK^{dp}}{(m!)^p}
\left(C_B\abs{b}_1\right)^{mp}e^{pC_B\abs{b}_1}\int_{\TT_L^d}\abs{f}^p\\
&\leq \frac{(8mnK)^{dp}(C_B\abs{b}_1)^{mp}e^{pC_B\abs{b}_1}}{(m!)^p}\int_{\TT_L^d}\abs{f}^p.
\end{split}
\end{equation}

Substituting the above inequality into \eqref{eq:eq:sum2} yields
\begin{align}
\nonumber
\int_{\TT_L^d}\abs{f}^p&\leq\Big(\frac{C_6^d}{\gamma}\Big)^{pnm-(p-1)}\int_{\TT_L^d\cap S}\abs{f}^p
\\ \label{eq:eq:24}
&\text{\hspace{20mm}}	+2^{p-1}\frac{(8mnK)^{dp}(C_B\abs{b}_1)^{mp}
e^{p C_B\abs{b}_1}}{(m!)^p}\Big(\frac{C_5^d}{\gamma}\Big)^{pnm}\int_{\TT_L^d}\abs{f}^p
\\ \nonumber
&\leq\Big(\frac{C_6^d}{\gamma}\Big)^{pnm-(p-1)}\int_{\TT_L^d\cap S}\abs{f}^p+
\frac{(\abs{b}_1)^{mp} e^{pC_B \abs{b}_1}}{{m^{mp}}}\Big(\frac{C_7^d}{\gamma}\Big)^{pnm}\int_{\TT_L^d}\abs{f}^p,
\end{align}
where the last inequality follows from Stirling's formula and the fact that $t\leq 2^t$.

Now, we choose $m$ such that
\[
\frac{(\abs{b}_1)^{mp} e^{pC_B \abs{b}_1}}{{m^{mp}}} \Big(\frac{C_7^d}{\gamma}\Big)^{pnm}
\leq \frac{1}{2}.
\]
For example, the choice
\begin{equation}\label{eq:m-choice}
m=\lceil 2 \, e \, C_B\left(\frac{C_7^d}{\gamma} \right)^n\abs{b}_1\rceil \leq 1+\left(\frac{C_8^d}{\gamma} \right)^n\abs{b}_1
\end{equation}
with $C_8=2\, e \, C_B\,  C_7 \geq 3$ fulfils the condition.

Then, we conclude
\begin{equation}\label{eq:eq:25}
\int_{\TT_L^d}\abs{f}^p
\leq
2 \left(\frac{C_6^d}{\gamma}\right)^{p\left(\frac{C_{9}^d}{\gamma}\right)^n\abs{b}_1+pn-(p-1)}\int_{\TT_L^d\cap S}\abs{f}^p.
\end{equation}

If $p=\infty$ the proof follows the same steps with obvious modifications (i.e. no passage to $L^p$-norm and
taking the maximum over all $\Lambda$ instead of summing over them).

In the general case when $L\in\RR^d_+$, the vector $a=(a_1,\ldots,a_d)$ has components $a_j\leq 2\pi L_j$ for all $j=1,\ldots,d$,
and the $J_l$ are disjoint, the same scaling argument as in Section \ref{s:proofLS} yields
\begin{equation}\label{eq:final}
\int_{\TT_L^d}\abs{f}^p
\leq\left(\frac{C_6^d}{\gamma}\right)^{p\left(\frac{C_{9}^d}{\gamma}\right)^n a\cdot b +pn-(p-1)}\int_{\TT_L^d\cap S}\abs{f}^p.
\end{equation}

To conclude the proof, consider the case when the $d$-dimensional rectangles $J_l$ are not disjoint.
We proceed by induction on the number of  $d$-dimensional rectangles  $n$.
If $n=1$ the result is true by either Theorem \ref{thm:1} or equation \eqref{eq:final}.

Let us assume that \eqref{eq:md-LogSer-torus-2} is true for $n\leq N$ and let us consider the case when $n=N+1$.

If the $d$-dimensional rectangles $J_1, \ldots, J_{N+1}$ are all disjoint, 
then the theorem is true by equation \eqref{eq:final}.

If among $J_1, \ldots, J_{N+1}$ there is a pair of $d$-dimensional rectangles which are not disjoint, 
it is possible to cover the $N+1$ $d$-dimensional rectangles of sidelengths $b_1,\ldots, b_d$
 with $N$  $d$-dimensional rectangles of sidelengths $3b_1,\ldots, 3b_d$.
Consequently
\begin{align}\label{eq:eq:26}
\norm{f}{L^p(\TT_L^d)}&
\leq\Big(\frac{\tilde{c}^d}{\gamma}\Big)^{\big(\frac{\tilde{c}^d}{\gamma}\big)^N 3  a \cdot b
+N-\frac{(p-1)}{p}}\norm{f}{L^p(S\cap\TT_L^d)}\notag\\
&\leq \Big(\frac{\tilde{c}^d}{\gamma}\Big)^{\big(\frac{\tilde{c}^d}{\gamma}\big)^{N+1} a \cdot b
+N+1-\frac{(p-1)}{p}}\norm{f}{L^p(S\cap\TT_L^d)},
\end{align}
where the first inequality follows by induction hypothesis and the second one is due to the fact that
$3\leq \tilde{c}\leq \frac{\tilde{c}^d}{\gamma}$. Hence, the proof is completed.

\begin{remark}
The quantities $C_B, K, C_1, C_3,C_4,C_5,C_6,C_7,C_8,C_9,
\tilde{C_3},\tilde{C_4}$, and $ \tilde{c}$ appearing in this section are
all numerical constants.
\end{remark}

\section{Proof of Theorem \ref{thm:sufficiency-necessity} and Remark \ref{rmk:sufficiency-necessity}} \label{sec:proof-T-2.3}

We recall that in the following $f$ is identified with its periodic extension on $\RR^d$.

$(i)\Rightarrow (ii)$ and $(ii')$: Let $S$ be a thick set with parameters $\gamma$ and $a$ and set $L_0:=\frac{1}{2\pi}\max_{j=1,\ldots, d} a_j>0$.
Let now $f$ be a periodic function with periodicity cell $[0,2\pi L_1]\times\ldots\times[0,2\pi L_d]$ with $\min_{j=1,\ldots,d} L_j\geq L_0$
and such that $\supp\widehat{f}\subset J$ (or $\supp\widehat{f}\subset J_1 \cup \ldots \cup J_n$, respectively).
Let now $h\in\RR^d$ and define $f_h:= f(\cdot -h)$. Then $\chi_{\TT_L^d} f_{h}$ is in $L^p(\TT_L^d)$
and its Fourier coefficients are again contained in $J$ (or in $J_1 \cup \ldots \cup J_n$, respectively).
Moreover, if $S$ is thick, so is $S+h$.
By change of variables and applying Theorem \ref{thm:1} (or Theorem \ref{thm:2} respectively) to $f_h$ and $S+h$ we obtain
\begin{align*}
 \norm{f}{L^p(\TT_L^d-h)}=\norm{f_{h}}{L^p(\TT_L^d)}
 \leq
 C\norm{f_{h}}{L^p((S+h)\cap\TT_L^d)}
 =
 C\norm{f}{L^p(S\cap(\TT_L^d-h))},
 \end{align*}
 where $C=C(\gamma, a\cdot b, d)$ is the constant in Theorem \ref{thm:1} (or Theorem \ref{thm:2} respectively).

 \medskip

 $(ii)\Rightarrow (i)$: Let $L_0>0$ be as in $(ii)$. It is sufficient to consider the case $L_1=\ldots=L_d=L_0$.
 Let $f\in L^p(\TT_L^d)$ be a periodic function with periodicity cell $\TT_L^d=([0,2\pi L_0])^d$
 and such that $\supp\widehat{f}\subset J$.
 Since for any $h\in \RR^d$, the sets $\TT_L^d$ and $\TT_L^d-h$ are periodicity cells of the same lattice, it holds $\norm{f}{L^p(\TT_L^d)}=\norm{f}{L^p(\TT_L^d -h)}$.

Using equation \eqref{eq:extended-LS-torus-2}, and Lemma \ref{lem:L-infinity-L-p-estimate} we obtain
\begin{align*}
 \norm{f}{L^p(\TT_L^d )}^p
 &=\norm{f}{L^p(\TT_L^d -h)}^p
    \leq C^p\int_{S\cap (\TT_L^d -h)}\abs{f(x)}^p \dd x \\
 & \leq C^p\abs{S\cap (\TT_L^d -h)}\sup_{x\in (\TT_L^d -h)} \abs{f(x)^p}
   = C^p\abs{S\cap (\TT_L^d -h)}(\sup_{y\in \TT_L^d } \abs{f(y)})^p\\
 & \leq C^p\abs{S\cap (\TT_L^d -h)} (2\pi)^{d(p-1)} \prod_{j=1}^d \frac{(2L_0 b_j +1)^{p}}{L_0}   \norm{f}{L^p(\TT_L^d)}^p\\
\end{align*}
This implies that for all $h\in\RR^d$
\begin{equation*}
 \abs{S\cap (\TT_L^d -h)} \geq \left(C^p  (2\pi)^{d(p-1)} \prod_{j=1}^d \frac{(2L_0 b_j +1)^{p}}{L_0}   \right)^{-1}=:\gamma>0,
\end{equation*}
where the constant $\gamma$ on the right hand side does not depend on $h$. Therefore, $S$ is a $(\gamma,a)$-\emph{thick}  set, where $a_1=\ldots=a_d=2\pi L_0$.

\section{Proof of Theorem \ref{thm:small-energy-interval}}
\label{s:nontrivial-V}

We will use here results of several papers which study observability for the time-dependent Schr\"odinger equation on $L^2(\TT_L^d)$
in the case that the `observability set' $S\neq \emptyset$ is an arbitrary open subset of $\TT_L^d$.
This has been studied for sufficiently regular $L^\infty(\TT^d)$-potentials on the standard size torus $\TT^d$ in any
dimension by Anantharaman and Macia in \cite{AnantharamanM-14}, for $C^\infty(\TT_L^2)$-potentials in dimension $d=2$
by Burq and Zworski in \cite{BurqZ-12}, and for $L^2(\TT_L^d)$-potentials ($d\in\{1,2\}$)
by Bourgain, Burq and Zworski in \cite{BourgainBZ-13},
which was again extended to dimension $d=3$ by Bourgain in \cite{Bourgain-14}.
We spell out some of their results which are of direct interest to this section.

\begin{theorem}[Anantharaman, Macia \cite{AnantharamanM-14}]\label{thm:AM}
Assume that $V\colon \TT^d\to \RR$ is in $L^\infty(\TT^d)$ and that its set of discontinuities  has measure zero.
Then, for every  open set $\emptyset \neq S\subset\TT^d$ and every $T>0$ there exists a constant ${K}={K}(V,T,S)>0$ such that
\begin{equation} \label{eq:observability-unitary-evolution}
\norm{f}{L^2(\TT^d)}^2\leq {K}\int_0^T\norm{e^{-t i H }f}{L^2(S)}^2\dd t,
\text{ for every  $f\in L^2(\TT^d)$.}
\end{equation}
Here  $H=-\Delta+V$ denotes the Schr\"odinger operator on $\TT^d$.
\end{theorem}
\begin{theorem}[Bourgain, Burq, Zworski \cite{BourgainBZ-13}, Bourgain \cite{Bourgain-14}] \label{thm:BBZ}
Let $d\in \{1,2,3\}$, $V\in L^2(\TT^d)$, and  $H=-\Delta+V$.
Then, for every  open set $\emptyset \neq S\subset\TT^d$ and every $T>0$ there exists a constant ${K}={K}(V,T,S)>0$ such that
inequality \eqref{eq:observability-unitary-evolution} holds.
\end{theorem}

If $f\in L^2(\TT^d)$ is an eigenfunction, i.e.~$-\Delta f+Vf= \lambda f $ for some real $\lambda \in \RR$,
then we have $e^{-it H}f= e^{-it \lambda}f$ and so inequality \eqref{eq:observability-unitary-evolution} yields
\begin{align}\label{eq:eq:2}
\norm{f}{L^2(\TT^d)}^2\leq {K}\int_0^T\norm{e^{-t i \lambda }f}{L^2(S)}^2\dd t
= {K}\int_0^T\norm{f}{L^2(S)}^2\dd t={K}T  \norm{f}{L^2(S)}^2,
\end{align}
which gives an instance of the bound \eqref{eq:question-E-independent-UCP} in the case of $f$ being a pure eigenfunction.
Note that the parameter $T$ is independent of the eigenvalue problem $Hf=\lambda f$, so one could optimize ${K}(V,T,S)\cdot T$
over $T>0$.

For linear combinations of eigenfunctions this bound can be strengthened somewhat, as was pointed out to us by N.~Anantharaman \cite{Anantharaman-16}
and independently by an anonymous referee. Indeed, it is possible to positively answer Question \ref{question} if the energy length $w$
is sufficiently small.
The proof of this statement is inspired by an argument in Section 2 of \cite{RamdaniTTT-05} and uses the following lemma,
which we state for the operator $H$ and whose general version and proof can be found in \cite[Lemma 5.2]{Miller-05c}.

\begin{lemma}\label{lemma:miller}
 Let $T>0$ and let $S\neq \emptyset $ be a open subset of $\TT^d$.
 Assume that there exists a constant $K=K(S,T,V)$ such that
\eqref{eq:observability-unitary-evolution} holds.
 Then,
\begin{equation}\label{eq:miller-2}
 \norm{f}{L^2(\TT^d)}^2
 \leq KT^3 \norm{(iH -i\lambda)f}{L^2(\TT^d)}^2
 + 2KT \norm{f}{L^2(S)}^2 \quad \forall \, f\in \mathcal{D}(H), \, \forall \, \lambda\in\RR.
\end{equation}
\end{lemma}

\begin{proposition}\label{prop:high-energy-interval}
Let $T>0$, $\emptyset \neq S\subset\TT^d$ be an open  subset.
Assume that $V$ is a potential and $K=K(S,T,V)$ a constant such that
\eqref{eq:observability-unitary-evolution} holds.
Let $w:=(2K T^3)^{-1/2}$.
Then, for all $f\in \bigcup_{E\in \RR} \Ran(\chi_{[E-w,E]}(H))$ we have
\begin{equation*}
\norm{f}{L^2(\TT^d)}^2 \leq 4KT \norm{f}{L^2(S)}^2.
\end{equation*}
\end{proposition}

\begin{proof}
Let $E\in\RR$ and let $f\in L^2(\TT^d)$ be a linear combination of normalized eigenfunctions $\phi_k$ of $H$ relative to eigenvalues in
$J_{E}:=[E-w, E]$, that is, $f(x)=\sum_{E_k\in J_{E}}\alpha_k\phi_k(x)$.
For such $f$ we have
\begin{align*}
 \norm{(iH -i E)f}{L^2(\TT^d)}^2 & = \norm{\sum_{E_k\in J_{E}}i(E_k-E)\alpha_k\phi_k}{L^2(\TT^d)}^2 \\
& = \sum_{E_k\in J_{E}}\abs{(E_k-E)\alpha_k}^2
\leq w^2\norm{f}{L^2(\TT^d)}^2
= (2K T^3)^{-1}\norm{f}{L^2(\TT^d)}^2.
\end{align*}
Lemma \ref{lemma:miller} yields
\begin{multline}\label{eq:step-1}
 \norm{f}{L^2(\TT^d)}^2
 \leq K T^3 \norm{(iH -i\lambda)f}{L^2(\TT^d)}^2 + 2TK \norm{f}{L^2(S)}^2\\
 \leq \frac{\norm{f}{L^2(\TT^d)}^2 }{2}  + 2TK \norm{f}{L^2(S)}^2
 \end{multline}
which implies the claim.
\end{proof}

\begin{proof}[Proof of Theorem \ref{thm:small-energy-interval}]
By the hypothesis on the potential and Theorem \ref{thm:AM} or \ref{thm:BBZ}
inequality \eqref{eq:observability-unitary-evolution} holds
with some constant ${K}={K}(T,S,V)$.
Let $w_T:=(2K T^3)^{-1/2}$ as above and $\kappa_T(S,V):=4 {K}(T,S,V) T$.
Pick $f\in \bigcup_{E\in\RR} \Ran(\chi_{[E-w, E]}(H))$.
Then by Proposition \ref{prop:high-energy-interval}
\begin{equation*}
\norm{f}{L^2(\TT^d)}^2
\leq \kappa_T \norm{f}{L^2(S)}^2.
\end{equation*}
Fixing a value for $T>0$ gives now the result. In specific applications the choice of $T$ can be used to optimize the constants $w_T$ and/or $\kappa_T$.
\end{proof}

\subsubsection*{Acknowledgements}
This work has been partially supported by the DFG under grant
VE 253/7-1 \emph{Multiscale version of the Logvinenko-Sereda Theorem}.
Part of this work was done while the authors were visiting the Hausdorff Research Institute for Mathematics
during the Trimester Program \emph{Mathematics of Signal Processing}.
The authors would like to thank Nalini Anantharaman and Angkana R\"uland for stimulating discussions and interest in this project,
Matthias T\"aufer and Albrecht Seelmann for commenting an earlier version of this manuscript and for clarifying discussions,
the anonymous referees whose remarks have stimulated several additions to the paper, and the editor for the handling of the paper.

\def\polhk#1{\setbox0=\hbox{#1}{\ooalign{\hidewidth
  \lower1.5ex\hbox{`}\hidewidth\crcr\unhbox0}}}

\end{document}